\documentclass[11pt,oneside,leqno]{amsart}
\usepackage{amsxtra}
\usepackage{amsopn}
\usepackage{color}
\usepackage{amsmath,amsthm,amssymb}
\usepackage{amscd}
\usepackage{amsfonts}
\usepackage{latexsym}
\usepackage{verbatim}

\definecolor{verde}{rgb}{0.5,.7,.2}

\theoremstyle{plain}
\newtheorem{theorem}{Theorem}[section]
\newtheorem{definition}[theorem]{Definition}
\newtheorem{lemma}[theorem]{Lemma}
\newtheorem{proposition}[theorem]{Proposition}
\newtheorem{corollary}[theorem]{Corollary}
\newtheorem{remark}[theorem]{Remark}
\newtheorem{example}[theorem]{Example}
\newtheorem{question}[theorem]{QUESTION}
\newtheorem{remark-question}[section]{Remark-Question}
\newtheorem{conjecture}[section]{Conjecture}

\newcommand\R{{\mathbb R}}

\newcommand\frn{{\mathfrak n}}

\begin{document}
\title[Laplacian flow of closed $\mathrm{G}_2$-structures inducing nilsolitons]{Laplacian flow of closed $\mathrm{G}_2$-structures\\
inducing nilsolitons}
\date{\today}
\author{Marisa Fern\'andez, Anna Fino and V\'ictor Manero}
\maketitle
\begin{abstract}
We study the existence of left invariant closed $G_2$-structures defining 
a Ricci soliton metric on simply connected nonabelian nilpotent Lie groups. 
 For each one of these $G_2$-structures, we show long time existence and
uniqueness of solution for the Laplacian flow on the noncompact 
manifold.  Moreover, considering the  Laplacian flow  on the associated Lie algebra as a bracket flow on $\R^7$ in a similar way as in \cite{La4} we prove that   the underlying metrics  $g(t)$ of the solution converge smoothly, up to pull-back by time-dependent 
diffeomorphisms,  to a flat metric, uniformly on compact sets in  the nilpotent Lie group, as $t$ goes to infinity.
 \end{abstract}


\begin{section}{Introduction}

A $G_2$-structure on a $7$-dimensional manifold $M$ can be characterized by the existence 
of a globally defined 3-form $\varphi$,
which is called  the  $G_2$ form  or the fundamental  $3$-form
and it can be described locally as
$$
\varphi=e^{127}+e^{347}+e^{567}+e^{135}-e^{146}-e^{236}-e^{245},
$$ 
with respect to some local basis $\{e^1,\dotsc, e^7\}$ of the $1$-forms on $M$.

There are many different $G_2$-structures attending to the behavior of the exterior derivative 
of the $G_2$ form \cite{Br-0, FernandezGray}.
In the following, we will focus our attention on {\em closed $G_2$-structures}
which are characterized by the 
closure of the $G_2$ form.

The existence of a $G_2$ form $\varphi$ on a manifold $M$ induces a Riemannian metric $g_{\varphi}$ on $M$
given by
\begin{equation}\label{metric}
g_{\varphi} (X,Y) vol=\frac 16 \iota_X\varphi \wedge \iota_Y\varphi \wedge \varphi,
\end{equation}
for any vector fields $X, Y$ on $M$, where $vol$ is the volume form on $M$.

By \cite{Br,CI} a closed $G_2$-structure on a compact manifold cannot induce an Einstein metric, 
unless the induced metric has holonomy contained in $G_2$.  It is still an open problem to see if 
the same property holds on noncompact  manifolds. For the homogeneous case, a negative answer 
has been recently given in \cite{FFM}. Indeed, we showed that if a solvable Lie algebra has a 
closed $G_2$-structure then the induced inner product is Einstein if and only if it is flat.

Natural generalizations of Einstein metrics are given by Ricci solitons, which  
have been introduced by Hamilton in \cite{Hamilton}. 
 A natural question is thus to see if a closed $G_2$-structure 
 on a noncompact manifold induces a (non-Einstein) Ricci soliton  
metric.
In this paper we  
give a positive answer to this question, showing that there exist $7$-dimensional  
simply connected nonabelian nilpotent 
Lie groups with a closed $G_2$-structure 
which determines a left invariant Ricci soliton metric. 

All known examples of nontrivial homogeneous Ricci solitons are left invariant metrics 
on simply connected solvable Lie groups, whose Ricci operator satisfies the condition 
$$Ric(g) =  \lambda I + D, $$ for some $\lambda \in \R$ and some derivation $D$ of the 
corresponding Lie algebra. The left invariant metrics satisfying the previous condition 
are called {\em nilsolitons} if the Lie groups are nilpotent \cite{La01}. 
Not all nilpotent Lie groups admit nilsoliton metrics, but if a nilsoliton 
exists, then it is unique up to automorphism and scaling \cite{La01}. The nilsolitons  metrics are strictly related to left invariant Einstein metrics on solvable Lie groups. Indeed, by \cite{La1}, a simply connected nilpotent Lie group $N$ admits a  
nilsoliton metric if and only if its Lie algebra $\mathfrak n$ is an Einstein nilradical, 
which means that $\frak n$ has
an inner product $\langle \cdot, \cdot \rangle$ such that there is a 
metric solvable extension of 
$(\mathfrak{n}, \langle \cdot, \cdot \rangle)$ which is Einstein.
According to \cite{La2,He}, such an Einstein metric has to be of standard type and 
it is unique, up to isometry and scaling.

Seven dimensional 
nilpotent Lie algebras admitting a closed $G_2$-structure have been recently classified in \cite{CF}, showing 
that there are twelve isomorphism classes, including the abelian case which has a trivial
nilsoliton because it is flat. 
A classification of  $7$-dimensional nilpotent Lie algebras admitting a nilsoliton has  been recently  given in \cite{Fe3},
but the explicit expression of the nilsoliton is not written in all the cases that we need.
 
Using the  classification in \cite{CF} and Table 1 in \cite{Fe1}, 
we have  that, up to isomorphism,  there is a unique 
nilpotent Lie algebra with a closed $G_2$ form but not admitting
nilsolitons.  It turns out that  all the other ten nilpotent Lie algebras have a nilsoliton, and we can determine
explicitly the nilsoliton except for  the Lie algebra
$\frak n_{10}$  which  is  4-step nilpotent
(see also \cite{Fe1,Fe2,Fe3}). 
In Proposition \ref{nilsoliton-NO-closed G2}
we prove that the Lie algebra $\frak n_{i}$ $(i=3, 5, 7, 8, 11)$ has a nilsoliton 
but no closed $G_2$-structure
inducing the nilsoliton.
Moreover, as we mentioned before, the existence of a nilsoliton on the Lie algebra $\frak n_{10}$
was shown in \cite[Example 2]{Fe1},
but we cannot explicit  its nilsoliton.
Therefore, it remains open the question of whether the Lie algebra $\frak n_{10}$ admits a
closed $G_2$ form inducing a nilsoliton or not.
This is the reason why the result of 
Theorem \ref{nilsoliton-closed G2} is restricted to
$s$-step nilpotent Lie algebras, with $s=2, 3$. In fact, in Theorem \ref{nilsoliton-closed G2}, 
we show that, up to isomorphism, there are 
exactly four $s$-step nilpotent Lie algebras $(s=2, 3)$ with a closed $G_2$ form defining a nilsoliton. 

The Ricci flow became a very important issue in Riemannian geometry and 
has been deeply studied. The same techniques are also useful in the 
study of the flow involving other geometrical structures, like for example, 
the K\"{a}hler Ricci flow that was studied by Cao in \cite{Cao}.

For any closed $G_2$-structure on a manifold $M$, in 
\cite{Br} Bryant introduced 
a natural flow, the so-called {\em Laplacian flow}, given by
$$
\left \{   \begin{array}{l} \frac{d}{dt} \varphi (t) = \Delta_t  \varphi(t),\\[3pt]
\varphi(0)=\varphi_0,
\end{array}
\right.
$$
where $\varphi (t)$ is a closed $G_2$ form on $M$, and $\Delta_t$ is the Hodge
Laplacian operator of the metric determined by $\varphi(t)$. If the initial $3$-form  
$\varphi_0$ is closed, then a solution $\varphi(t)$ of the flow remains closed, 
and the de Rham cohomology
class $[\varphi(t)]$ is constant in $t$.
The short time existence and uniqueness of solution for the Laplacian flow of any
closed $G_2$-structure, 
on a compact manifold $M$, has been proved by Bryant and Xu in the unpublished paper \cite{Br-Xu}.
Also, long time existence and convergence of the Laplacian flow
starting near a torsion-free $G_2$-structure was proved in  the unpublished paper \cite{Xu-Ye}.

In Section \ref{sectLaplacian}  (Theorem \ref{Lap-flow: N2}, Theorem \ref{Lap-flow: N4}, Theorem \ref{Lap-flow: N6} and Theorem \ref{Lap-flow: N12}) we   show long time existence of the solution for the 
Laplacian flow on the four nilpotent Lie groups admitting an invariant 
closed $G_2$-structure which determines the nilsoliton
(see Theorem \ref{nilsoliton-closed G2}).

To our knowledge, these are the first examples of noncompact manifolds
having a closed $G_2$-structure with long time existence of solution.

Since the  Laplacian flow is invariant by diffeomorphisms and the initial $G_2$-form $\varphi_0$  
is invariant, the solution $\varphi (t)$ of the Laplacian flow has to be also invariant. Therefore, 
we show that the Laplacian flow is equivalent to a system of ordinary differential equations which 
admits a unique solution. We prove that the solution for the four manifolds is defined 
for any $t\in [0, + \infty)$.  Moreover, considering the  Laplacian flow on the associated Lie algebra as a bracket flow on $\R^7$, in a similar way as Lauret 
did in \cite{La4} for the Ricci flow, we show that the underlying metrics $g(t)$ of the solution converge smoothly, up to pull-back by time-dependent diffeomorphisms,  to a flat metric, uniformly on compact sets in the nilpotent Lie group as $t$  goes to infinity.  Indeed, 
by \cite[Proposition 2.1]{La4}
the convergence of  the  metrics   in the   ${\mathcal C}^{\infty}$  uniformly on compact sets in $\R^7$ is equivalent to the convergence  of the nilpotent Lie  brackets $\mu_t$  in the algebraic subset of nilpotent Lie brackets ${\mathcal N}  \subset (\Lambda^2 \R^7) ^* \otimes \R^7$  with the usual vector space topology. 

\end{section}


\begin{section}{Preliminaries on nilsolitons}

In this section, we recall some definitions and results about 
{\em nontrivial homogeneous Ricci soliton metrics} and, in particular, 
on {\em nilsolitons}. For more details, 
see for instance \cite{Chow}, \cite{La01} and \cite{Jablonski}.

A complete Riemannian metric $g$ on a manifold $M$ is said to be a
{\em Ricci soliton} if its Ricci curvature tensor $Ric(g)$ satisfies the following condition
\begin{equation*}\label{ricci-soliton}
Ric(g) =  \lambda g + {\mathcal{L}}_X g, 
\end{equation*}
for some real constant $\lambda$ and a complete vector field $X$ on $M$, where 
${\mathcal{L}}_X$ denotes  the
Lie derivative with respect to $X$.
If in addition $X$ is the gradient vector field of a smooth function $f\colon M \rightarrow \R$, 
then the Ricci soliton is said to be of {\em gradient type}. Ricci solitons are called {\em expanding}, 
{\em steady} or {\em shrinking} depending on 
whether $\lambda  < 0$, $\lambda = 0$ or $\lambda  > 0$, respectively.

In the next section we will focus our attention on nilsolitons, that is,
a particular type of nontrivial homogeneous Ricci
soliton metrics. 

A Ricci soliton metric $g$ on $M$
is called {\em trivial} if $g$ is an Einstein metric or $g$ is the product 
of a homogeneous Einstein metric with the Euclidean metric; and
$g$ is said to be {\em homogeneous} if its
isometry group acts transitively on $M$, and hence
$g$ has bounded curvature \cite{La3}. 

In order to characterize the nontrivial homogeneous Ricci
soliton metrics, we note that any homogeneous steady or shrinking   Ricci 
soliton metric $g$ of gradient type
is trivial.  Indeed, if $g$ is steady, one can check 
that $g$ is Ricci flat, and so by \cite{AK} $g$ must be flat. 
If $g$ is shrinking, then
by the results in \cite[Theorem 1.2]{Naber} and in \cite{PW},
$(M, g)$ is isometric to a quotient of $P \times \R^k$, 
where $P$ is some homogeneous Einstein manifold with positive scalar curvature.  
Now, we should notice that this last result for shrinking
homogeneous Ricci soliton metrics is also true 
for homogeneous Ricci solitons of gradient type \cite{PW}.
Moreover, if a homogeneous Ricci soliton $g$ on a manifold $M$ is expanding, then by \cite{I}
$M$ must be noncompact; and from \cite{Perelman} all Ricci solitons 
(homogeneous or nonhomogeneous) on a
compact manifold are of gradient type. 
Therefore, as it was noticed by Lauret in  \cite{La3} we have the following
\begin{lemma}$($\cite{La3}$)$ 
Let $g$ be a nontrivial homogeneous Ricci soliton on a 
manifold $M$. Then, $g$ is expanding and it cannot be of gradient type. Moreover, $M$ is noncompact.
\end{lemma} 

All known examples of nontrivial homogeneous Ricci solitons are
left invariant metrics on simply connected solvable Lie groups whose Ricci
operator is a multiple of the identity modulo derivations,
and they are called \emph{solsolitons} or, in the nilpotent case, \emph{nilsolitons}.

Let $N$ be a simply connected nilpotent Lie group, and denote by $\frak n$ its Lie algebra.
A left invariant metric $g$ on $N$ is  called a \emph{Ricci nilsoliton metric} 
(or simply \emph{nilsoliton metric}) if its Ricci endomorphism 
$Ric(g)$ differs from a derivation $D$ of 
$\frak n$ by a scalar multiple of the identity map $I$, i.e.  
if there exists a real number $\lambda$ such  that  
\begin{equation} \label{condition-derivation} 
Ric(g)= \lambda I + D.
\end{equation}
Clearly, any left invariant metric which satisfies \eqref{condition-derivation} is 
automatically a Ricci soliton.

Nilsoliton metrics have properties that make them preferred left invariant metrics 
on nilpotent Lie groups in the absence of Einstein metrics. Indeed,  
nonabelian nilpotent Lie groups do not admit left invariant Einstein metrics (\cite{Milnor}).

From now on, we will always identify a left invariant metric on a Lie group 
$N$ with an inner product 
$\langle \cdot, \cdot \rangle_{\mathfrak{n}}$ on the Lie algebra 
$\mathfrak{n}$ of $N$.  
A Lie algebra $\mathfrak{n}$ endowed with an inner product 
is usually called in the literature a {\em metric Lie algebra} and is denoted as 
the pair $(\mathfrak{n},\langle \cdot, \cdot \rangle_{\mathfrak{n}})$. 

We  will say that a metric nilpotent Lie algebra $(\frak n, \langle \cdot, \cdot \rangle _{\frak n})$ is 
a \emph{nilsoliton} if there exists a real number $\lambda$ 
and a derivation $D$ of $\frak n$ such that 
\begin{equation}\label{nilsoliton:lie-algebra}
Ric(\frak n, \langle \cdot, \cdot \rangle _{\frak n}) = \lambda I + D.
\end{equation}
Not all nilpotent Lie algebras admit nilsoliton inner products, but if a nilsoliton 
inner product exists, then it is unique up to automorphism and scaling \cite{La01}.
A computational method for classifying  nilpotent Lie algebras
having a nilsoliton inner product in a large subclass of the set
of all nilpotent Lie algebras, has been recently introduced in \cite{KP}.
By Lauret's results it turns out that nilsoliton metrics on simply connected nilpotent 
Lie groups $N$ are strictly related to Einstein metrics on the so-called solvable  
rank-one extensions of $N$. 

\begin{definition} 
\emph{Let $(\mathfrak{n},\langle \cdot, \cdot \rangle)$ be a metric nilpotent Lie algebra. 
A} metric solvable extension 
of \emph{$(\mathfrak{n},\langle \cdot, \cdot \rangle)$ is
a metric solvable Lie algebra 
$(\mathfrak{s}=\mathfrak{n}\oplus \mathfrak{a}, \langle \cdot, \cdot \rangle_{\mathfrak{s}})$
such that 
$\mathfrak{n}=[\mathfrak{s}, \mathfrak{s}]$ and 
$\langle \cdot, \cdot \rangle_{\mathfrak{s}}|_{\mathfrak{n}\times \mathfrak{n}}=\langle \cdot, \cdot \rangle$. 
The metric solvable Lie algebra 
$(\mathfrak{s}, \langle \cdot, \cdot \rangle_{\mathfrak{s}})$ is} 
standard, \emph{or has} standard type,
\emph{if $\frak a$ is an abelian subalgebra of $\mathfrak{s}$; 
in this case, the dimension of $\mathfrak{a}$ 
is called the} rank of the metric solvable extension. 
\end{definition}

Heber showed in \cite{He} that a simply connected solvable Lie group admits 
at most one Einstein left invariant metric up to isometry and scaling. Moreover, 
he proved that the study of Einstein metrics 
on simply connected solvable Lie groups, of standard type,
can be reduced to the rank-one case, 
that is, $dim \, \mathfrak{a}=1$. 

\medskip

Recently, Lauret in \cite{La1} and \cite{La2} proved the following 

\begin{theorem} $($\cite{La1, La2}$)$ \label{Lauret Theorem}
Any Einstein metric solvable Lie algebra 
$(\mathfrak{s}, \langle \cdot, \cdot \rangle_{\mathfrak{s}})$ has to be of standard type.
Moreover, a simply connected nilpotent Lie group $N$ admits a  
nilsoliton metric if and only if its Lie algebra $\frak n$ is an Einstein nilradical, that is,
$\frak n$ possesses an inner product $\langle \cdot, \cdot \rangle$ such that 
$(\mathfrak{n}, \langle \cdot, \cdot \rangle)$ has a 
metric solvable extension which is Einstein.
\end{theorem}

\end{section}


\begin{section}{Nilsoliton metrics determined by closed $G_2$ forms} \label{sectclassif}

In this section we prove that, up to isomorphism, there
are only four (nonabelian)  $s$-step nilpotent Lie  groups   
$(s=2, 3)$ with a nilsoliton inner product determined by a 
left invariant closed $G_2$-structure.
We also show  that, up to isomorphism,  there is a unique $7$-dimensional nilpotent Lie group with
a left invariant closed $G_2$-structure but not having nilsoliton metrics.

Let $N$ be a $7$-dimensional simply connected nilpotent
Lie group with Lie algebra $\mathfrak{n}$. Then, a $G_{2}$-structure 
on $N$ is left invariant if and only if the corresponding
$3$-form is left invariant. Thus, a left invariant $G_{2}$-structure on 
$N$ corresponds to an element $\varphi$ of $\Lambda^3({\mathfrak{n}}^*)$ that 
can be written as 
\begin{equation*}\label{eqn:3-forma G2}
 \varphi=e^{127}+e^{347}+e^{567}+e^{135} -e^{236}-e^{146}-e^{245},
\end{equation*}
with respect to some coframe $\{e^1,\dotsc, e^7\}$ on ${\mathfrak{n}}^*$, and
we shall say that $\varphi$ defines a $G_{2}$-structure on $\mathfrak{n}$.  
A $G_{2}$-structure on $\mathfrak{n}$ is said to be {\em closed} if $\varphi$ is closed, i.e.
\[
d\varphi=0,\]
where $d$ denotes the Chevalley-Eilenberg differential on ${\mathfrak{n}}^*$. 

From now on, given a 
$7$-dimensional Lie algebra $\mathfrak{n}$ whose dual is spanned by 
$\{ e^1,\ldots ,e^7\}$, we will write $e^{ij}= e^i\wedge e^j$,
$e^{ijk}= e^i\wedge e^j\wedge e^k$, and so forth. 
Moreover, by  the notation
$$
\mathfrak{n}=(0,0,0,0,e^{12},e^{13},0),
$$
we mean that the dual space ${\mathfrak{n}}^*$ of the Lie algebra $\frak n$ has a fixed basis 
$\{e^1,\dotsc, e^7\}$ such that
$$
de^5=e^{12},\quad \;de^6=e^{13}, \quad \; d e^1 = d e^2 = d e^3 = d e^4 = d e^7=0.
$$
The classification of nilpotent Lie algebras admitting a closed $G_2$-structure is given in \cite{CF} as follows.
\begin{theorem}\label{classification}
Up to isomorphism, there are exactly $12$ nilpotent Lie algebras that admit a 
closed $G_{2}$-structure.
They are:
$$
\begin{array}{l}
\mathfrak{n}_1 = (0,0,0,0,0,0,0),\\
\mathfrak{n}_2=(0,0,0,0,e^{12},e^{13},0),\\
\mathfrak{n}_{3}=(0,0,0,e^{12},e^{13}, e^{23},0),\\
\mathfrak{n}_{4}=(0,0,e^{12},0,0,e^{13}+e^{24},e^{15}),\\
\mathfrak{n}_{5}=(0,0,e^{12},0,0,e^{13},e^{14}+e^{25}),\\
\mathfrak{n}_{6}=(0,0,0,e^{12},e^{13},e^{14},e^{15}),\\
\mathfrak{n}_{7}=(0,0,0,e^{12},e^{13},e^{14}+e^{23},e^{15}),\\ 
\mathfrak{n}_{8}=(0,0,e^{12},e^{13}, e^{23},e^{15}+e^{24},e^{16}+e^{34}),\\
\mathfrak{n}_9=(0,0,e^{12}, e^{13}, e^{23}, e^{15}+e^{24},e^{16}+e^{34}+e^{25}),\\
\mathfrak{n}_{10}=(0,0,e^{12},0,e^{13}+e^{24},e^{14},e^{46}+e^{34}+e^{15}+e^{23}),\\
\mathfrak{n}_{11}=(0,0,e^{12},0,e^{13},e^{24}+e^{23}, e^{25}+e^{34}+e^{15}+e^{16}-3e^{26}),\\
\mathfrak{n}_{12}=(0,0,0,e^{12},e^{23},-e^{13},2e^{26}-2e^{34}-2e^{16}+2e^{25}).
\end{array}
$$
\end{theorem}

Using Table 1 in \cite{Fe1} we  can determine which indecomposable Lie algebras $\mathfrak{n}_i$ $(4\leq i\leq 12)$
do not  have  nilsoliton inner products. 
Note that  the existence of nilsolitons
on $\mathfrak{n}_2$ and $\mathfrak{n}_3$ is not studied in \cite{Fe1} since they are decomposable.
Concretely, the correspondence between the indecomposable Lie  algebras of Theorem \ref{classification} and Table 1 in \cite{Fe1} is the following: 
\begin{equation*}
\begin{aligned}
&\frn_4 \cong 3.8, \qquad \frn_5 \cong 3.11\qquad \frn_6 \cong 3.20, \qquad \frn_7 \cong 2.39,\\
&\frn_8 \cong 2.5, \qquad \frn_9 \cong 1.1(iv), \qquad 	\text{ and } \qquad \frn_{10} \cong 1.3(i_1).
\end{aligned}
\end{equation*} 

Moreover, $\frn_{11}$ and $\frn_{12}$ are respectively isomorphic to the real form of $1.2 (i_{-3})$ and $3.1 (i_{2})$.
 In particular, we   have that  $\mathfrak{n}_9$ is the only $7$-dimensional nilpotent Lie algebra with a closed $G_2$ form but not admitting a nilsoliton.

\begin{remark} Note that the abelian Lie algebra $\mathfrak{n}_1$ admits as rank-one Einstein 
solvable extension the Lie algebra  $\frak s_1$ with structure equations 
$$
(ae^{18},ae^{28}, ae^{38}, ae^{48}, ae^{58}, ae^{68}, ae^{78},0),
$$
for some real number $a \neq 0$, 
and the nilsoliton inner product on $\frak n_1$ is trivial because it is flat.  
Since we are interested in nontrivial nilsoltons inner products, in the sequel when we refer to 
a nilpotent Lie algebra we will mean a nonabelian nilpotent Lie algebra.  \end{remark}

In order to classify the Lie algebras $\mathfrak{n}_i$ admitting a (nontrivial) nilsoliton 
but with no closed $G_2$ forms inducing the nilsoliton, we need to recall
 the following obstruction proved in \cite{CF} for the existence of a closed 
$G_2$-structure on a $7$-dimensional Lie algebra.

\begin{lemma}[\cite{CF}]\label{Obs1} 
Let $\mathfrak{g}$ be a 7-dimensional Lie algebra. If there is a non-zero 
$X\in\mathfrak{g}$ such that $(\iota_X\phi)^3=0$ (where $\iota_X$ denotes the contraction by
$X$) for every closed 3-form $\phi$ on $\mathfrak{g}$, then 
$\mathfrak{g}$ has  no closed $G_2$-structures. 
\end{lemma}

By \cite [Proposition 4.5]{Schulte}
if $\varphi$  is  a  $G_2$-structure on a $7$-dimensional Lie algebra  and we choose 
a vector $X \in \frak g$  of length one with respect to $g_{\varphi}$,  then on  the orthogonal 
complement  of  the span of  $X$  one has an $SU(3)$-structure given by 
 the 2-form $\alpha = \iota_X \varphi$  and the 3-form
$\beta = \varphi - \alpha \wedge \eta$, where $\eta=\iota_{X}(g_{\varphi})$. So 
in particular $\alpha \wedge  \beta =0$.

By using these results we can prove the following proposition

\begin{proposition}\label{nilsoliton-NO-closed G2}
The Lie algebra $\mathfrak{n}_i$ $(i=3,5,7,8,11)$ has a nilsoliton inner product
but  no closed $G_2$-structure inducing the nilsoliton inner product. 
\end{proposition}
\begin{proof}

To prove that $\mathfrak{n}_3$ has a nilsoliton, we consider  the Lie algebra $\mathfrak{n}_3$ defined by the equations given in Theorem \ref{classification}.
Let $\langle \cdot, \cdot \rangle_{\mathfrak{n}_3}$ be the inner product on $\frak n_3$
such that $\{ e^1, \ldots, e^7 \}$ is orthonormal. 
Then, $\langle \cdot, \cdot \rangle_{\mathfrak{n}_3}$ is a nilsoliton because  its Ricci tensor
$$Ric=diag\Big(-1,-1,-1,\frac{1}{2},\frac{1}{2},\frac{1}{2},0\Big)$$ 
satisfies \eqref{nilsoliton:lie-algebra}, 
for $\lambda =  -5/2$ and 
$$D=diag\Big(\frac{3}{2},\frac{3}{2},\frac{3}{2},3,3,3,\frac{5}{2}\Big).$$
Since the nilsoliton inner product is unique (up to isometry and scaling) it suffices to prove 
that there is no closed $G_2$ form on $\mathfrak{n}_3$ inducing such an inner product.

 Suppose that $\mathfrak{n}_3$ has  a closed $G_2$ form $\phi$ such that 
\begin{equation}\label{g2:nilsoliton-n3}
g_{\phi}= \langle \cdot, \cdot \rangle_{\mathfrak{n}_3}  = \sum_{i = 1}^7  (e^i)^2. 
\end{equation}
Thus, $g_{\phi}$ has to  satisfy
\begin{equation}\label{productorio}
\prod_{i=1}^7g_{\phi}(e_i, e_i)=1.
\end{equation}
A generic closed 3-form 
$\gamma$ on $\mathfrak{n}_3$ has the following expression
\begin{equation*}
\begin{aligned}
\gamma=&c_{123} e^{123}+c_{124} e^{124}+c_{125} e^{125}+c_{126} e^{126}+c_{127} e^{127}+c_{134} e^{134}+c_{135} e^{135}\\
&+c_{136} e^{136}+c_{137} e^{137}+c_{145} e^{145}+c_{146}e^{146}+c_{147} e^{147}+c_{156} e^{156}+c_{157} e^{157}\\
&+c_{167} e^{167}+c_{234} e^{234}+c_{235} e^{235}+c_{236} e^{236}+c_{237} e^{237}+c_{146} e^{245}+c_{246}e^{246}\\
&+c_{247} e^{247}+c_{256} e^{256}+c_{257} e^{257}+c_{267} e^{267}+c_{156} e^{345}+c_{256} e^{346}\\
&+\left(c_{257}-c_{167}\right) e^{347}+c_{356} e^{356}+c_{357}e^{357}+c_{367} e^{367},
\end{aligned}
\end{equation*}
where $c_{ijk}$ are arbitrary real numbers.
 Now, we show conditions on the coefficients $c_{ijk}$ so that $\phi = \gamma$ is a closed $G_2$ form such that
$g_{\phi}$ satisfies \eqref{g2:nilsoliton-n3}.
To this end, we apply the aforementioned result of \cite [Proposition 4.5]{Schulte} for $X=e_i$ $(1\leq i \leq 7)$ and so $\eta=e^i$  by \eqref{g2:nilsoliton-n3}.
For $X=e_1$, thus $\eta=e^1$,  we have
\begin{equation*}
\begin{aligned}
\alpha_1=&\iota_{e_1}\phi=c_{123}e^{23}+c_{124}e^{24}+c_{125}e^{25}+c_{126}e^{26}+c_{127}e^{27}+c_{134}e^{34}+c_{135}e^{35}\\
&+c_{136}e^{36}+c_{137}e^{37}+c_{145}e^{45}+c_{146}e^{46}+c_{147}e^{47}+c_{156}e^{56}+c_{157}e^{57}+c_{167}e^{67},
\end{aligned}
\end{equation*}
and
\begin{equation*}
\begin{aligned}
\beta_1=&\phi-\iota_{e_1}\phi\wedge e^1=c_{234} e^{234}+c_{235} e^{235}+c_{236} e^{236}+c_{237} e^{237}+c_{146} e^{245}+c_{246}e^{246}\\
&+c_{247} e^{247}+c_{256} e^{256}+c_{257} e^{257}+c_{267} e^{267}+c_{156} e^{345}+c_{256} e^{346}\\&
+\left(c_{257}-c_{167}\right) e^{347}+c_{356} e^{356}+c_{357}e^{357}+c_{367} e^{367},
\end{aligned}
\end{equation*}
 But, 
$
\alpha_1\wedge \beta_1=0
$
describes a system of 6 equations. Hence, after apply the result of \cite [Proposition 4.5]{Schulte} for $X=e_2, \dots, e_7$,
we obtain a system of 42 equations. 
This system and condition \eqref{g2:nilsoliton-n3}
imply that any closed $G_2$ form  on $\mathfrak{n}_3$ satisfying \eqref{productorio} is expressed as follows
\begin{equation}\label{trick-1}
\begin{aligned}
\phi =&c_{123} e^{123}+c_{145} e^{145}+c_{167} e^{167}+c_{246} e^{246}+c_{257} e^{257}\\
&+\left(c_{257}-c_{167}\right) e^{347}+c_{356} e^{356}.
\end{aligned}
\end{equation}
Because $\phi$ should be a closed $G_2$ form on $\mathfrak{n}_3$, at least for certain coefficients $c_{ijk}$, Lemma \ref{Obs1} implies that the coefficients appearing on \eqref{trick-1}
cannot vanish. In particular,
$c_{257}-c_{167}\not=0.$
 Now, denote by $G_{\phi}$ the matrix associated to the inner product on $\mathfrak{n}_3$ induced by the 3-form $\phi$ given by \eqref{trick-1}. Then, 
\eqref{g2:nilsoliton-n3} implies that
$G_{\phi}=I_7$, for some $c_{ijk}$ and then
\begin{equation} \label{sistema2}
S=G_{\phi}- I_7 =0,
\end{equation}
for those coefficients.
From now on, we denote by $S_{ij}$ the $(i,j)$ 
entry of the matrix $S$. 
One can check that the equations $S_{11} = S_{22} =S_{55} =0$ imply that
\begin{equation*}
c_{356}=\frac{1}{c_{145}c_{257}}, \quad c_{246}=-\frac{c_{145}c_{167}}{c_{257}}\quad \text{ and } \quad c_{123}=\frac{1}{c_{145}c_{167}}.
\end{equation*}
Therefore, the expression of $S_{66}$ becomes
\begin{equation*}
S_{66}=\frac{(c_{167}-c_{257})(c_{167}+c_{257})}{c_{257}^2},
\end{equation*}
and hence $c_{167}=\pm c_{257}$.
But 
we know that  $c_{167}\not=c_{257}$, and for $c_{167}=-c_{257}$, we have  
that $S_{33}=-c_{123}(c_{167}-c_{257})c_{356}$ and so 
$S\not=0$,  which is a contradiction with \eqref{sistema2}. This means that $\mathfrak{n}_3$ does not admit 
a closed $G_2$ form inducing the nilsoliton given by \eqref{g2:nilsoliton-n3}.

\medskip
To prove that $\mathfrak{n}_5$ has a nilsoliton, we consider the Lie algebra $\mathfrak{n}_5$ defined by the structure equations
$$\mathfrak{n}_5=(0,0,\sqrt{3}e^{12},0,0,2e^{13},e^{14}+\sqrt{3}e^{25}).$$
Consider the inner product $\langle \cdot, \cdot \rangle_{\mathfrak{n}_5}$ such that the basis $\{e^1,\dots , e^7 \}$ is orthonormal.
Then, its Ricci tensor satisfies
$$Ric=diag\Big(-4,-3,-\frac{1}{2},-\frac{1}{2},-\frac{3}{2},2,2\Big).$$
Actually,
$Ric=-\frac{13}{2}I_7+D,$
where $D$ is the derivation of $\mathfrak{n}_5$ given by
$$D=diag\Big(\frac{5}{2},\frac{7}{2},6,6,5,\frac{17}{2},\frac{17}{2}\Big)$$ 
and so $\langle \cdot, \cdot \rangle_{\mathfrak{n}_5} = \sum_{i = 1}^7  (e^i)^2$ is a nilsoliton inner product.

Since the nilsoliton inner product is unique (up to isometry and scaling) it  is sufficient  to prove that there is no closed $G_2$ form on $\mathfrak{n}_5$ inducing such an inner product. Suppose that $\mathfrak{n}_5$ has a closed $G_2$ form $\phi$ such that   $g_{\phi} = \langle \cdot, \cdot \rangle_{\mathfrak{n}_5}$.

A generic closed 3-form $\gamma$ on $\mathfrak{n}_5$ has the following expression
\begin{equation*}
\begin{aligned}
\gamma=&c_{123} e^{123}+c_{124} e^{124}+c_{125} e^{125}+c_{126} e^{126}+c_{127} e^{127}+c_{134} e^{134}+c_{135} e^{135}\\
&+c_{136} e^{136}+c_{137}e^{137}+c_{145} e^{145}+c_{146} e^{146}+c_{147} e^{147}+c_{156} e^{156}+c_{157} e^{157}\\
&+c_{167} e^{167}+c_{234} e^{234}+c_{235}e^{235}+c_{236} e^{236}+c_{237} e^{237}+c_{245} e^{245}+\frac{1}{2} c_{237} e^{246}\\
&+c_{247} e^{247}-\frac{1}{2} \sqrt{3} c_{137} e^{256}+ \sqrt{3} ( c_{345}- c_{147})e^{257}+c_{345} e^{345}-c_{167} e^{356}+c_{457} e^{457},
\end{aligned}
\end{equation*}
where $c_{ijk}$ are arbitrary real numbers. Now we show conditions on the coefficients $c_{ijk}$ so that $\phi=\gamma$ is a closed $G_2$ form such that  $g_{\phi} = \langle \cdot, \cdot \rangle_{\mathfrak{n}_5}$. Lemma \ref{Obs1} (applied for $X=e_7$) implies that
\begin{equation}\label{inequality}
c_{167}\,c_{237}\,c_{457}\neq 0.
\end{equation} 
Now, we denote by $G_{\phi}$ the matrix associated to the inner product on $\mathfrak{n}_5$ induced by the generic closed 3-form $\phi$. 
Then the condition $g_{\phi} = \langle \cdot, \cdot \rangle_{\mathfrak{n}_5}$ implies \eqref{sistema2} for some coefficients $c_{ijk}$.
From the equations $S_{66}=S_{77}=S_{67}=S_{37}=S_{46}=S_{33}=S_{36}=S_{47}=0$ we have that 
\begin{equation*}
\begin{aligned}
c_{237}&=\frac{2}{c_{167}^2}, \quad c_{457}=\frac{1}{2}c_{167}, \quad c_{236}=-2c_{247},   \quad c_{136}=-2c_{147},\\
  c_{345}&=0,   \qquad c_{134}=\frac{1}{2} c_{167},  \qquad c_{137}=2c_{146}, \qquad c_{234}=0.
\end{aligned}
\end{equation*}
 Therefore, $S_{44}=-\frac{3}{8}c_{167}^2c_{237}$ which by \eqref{inequality} cannot vanish and so $S \neq 0$, which is a contradiction with \eqref{sistema2}.

\medskip

Consider now the Lie algebra $\mathfrak{n}_7$ defined by the structure equations
$$\mathfrak{n}_7=\Big(0,0,0,e^{12},\frac{\sqrt{6}}{2}e^{13}, e^{14}+\frac{\sqrt{6}}{2}e^{23},\sqrt{2}e^{15}\Big).$$
Let $\langle \cdot, \cdot \rangle_{{\mathfrak{n}_7}}$ be the inner product on $\mathfrak{n}_7$ such that the basis $\{e^1,\dots,e^7\}$ is orthonormal.
Then, $\langle \cdot, \cdot \rangle_{\mathfrak{n}_7} = \sum_{i = 1}^7 (e^i)^2$ is a nilsoliton since 
$$Ric=\Big(-\frac{11}{4},-\frac{5}{4},-\frac{3}{2},0,-\frac{1}{4},\frac{5}{4},1\Big) = -4I_7+D,$$
where
$$D=diag\Big(\frac{5}{4},\frac{11}{4},\frac{5}{2},4,\frac{15}{4},\frac{21}{4},5\Big),$$
is a derivation of $\mathfrak{n}_7$.
As before, since the nilsoliton inner product is unique (up to isometry and scaling) it suffices to prove that there is no closed $G_2$ form on $\mathfrak{n}_7$ inducing such an inner product.

Suppose that $\mathfrak{n}_7$ has a closed $G_2$ form $\phi$ such that   $g_{\phi} = \langle \cdot, \cdot \rangle_{{\mathfrak{n}_7}}$. A generic closed 3-form $\gamma$ on $\mathfrak{n}_7$ has the following expression
\begin{equation*}
\begin{aligned}
\gamma=&c_{123} e^{123}+c_{124} e^{124}+c_{125} e^{125}+c_{126} e^{126}+c_{127} e^{127}+c_{134} e^{134}+c_{135} e^{135}\\
&+c_{136} e^{136}+c_{137}e^{137}+c_{145} e^{145}+c_{146} e^{146}+c_{147} e^{147}+c_{156} e^{156}+c_{157} e^{157}\\
&+c_{167} e^{167}+c_{234} e^{234}+c_{235} e^{235}+\left(\frac{\sqrt{6}}{2} c_{245}-\frac{\sqrt{6}}{2}  c_{146}\right) e^{236}+c_{237} e^{237}\\
&+c_{245} e^{245}+c_{246}e^{246}+\frac{\sqrt{2}}{2}c_{256} e^{247}+c_{256} e^{256}+\left(c_{167}+\frac{\sqrt{6}}{3} c_{347}\right) e^{257}\\
&+\left(\frac{\sqrt{6}}{2}
   c_{156}+\sqrt{2} c_{237}\right) e^{345}+\frac{\sqrt{6}}{2}c_{256} e^{346}+c_{347} e^{347}+\sqrt{2} c_{347} e^{356}+c_{357} e^{357},
\end{aligned}
\end{equation*}
where $c_{ijk}$ are arbitrary real numbers. Now, we show conditions on the coefficients $c_{ijk}$ so that $\phi=\gamma$ be a closed $G_2$ form such that  $g_{\phi} =\langle \cdot, \cdot \rangle_{\mathfrak{n}_7}$. Lemma \ref{Obs1} applied for $X=e_7$ implies that 
\begin{equation}\label{conditionn5}
c_{167} \neq 0.
\end{equation}
Now we apply the result of \cite [Proposition 4.5]{Schulte} for $X=e_i$  $(1 \leq i \leq 7)$ and so $\eta= e^i$ by \eqref{g2:nilsoliton-n3}. For $X=e_1$, we have
\begin{equation*}
\begin{aligned}
\alpha_1=&\iota_{e_1}\phi=c_{123} e^{23}+c_{124} e^{24}+c_{125} e^{25}+c_{126} e^{26}+c_{127} e^{27}+c_{134} e^{34}+c_{135} e^{35}\\
&+c_{136} e^{36}+c_{137}e^{37}+c_{145} e^{45}+c_{146} e^{46}+c_{147} e^{47}+c_{156} e^{56}+c_{157} e^{57}\\
&+c_{167} e^{67}\end{aligned}
\end{equation*}
and
\begin{equation*}
\begin{aligned}
\beta_1=&\phi-\iota_{e_1}\phi\wedge e^1=c_{234} e^{234}+c_{235} e^{235}+\left(\frac{\sqrt{6}}{2} c_{245}-\frac{\sqrt{6}}{2}  c_{146}\right) e^{236}+c_{237} e^{237}\\
&+c_{245} e^{245}+c_{246}e^{246}+\frac{\sqrt{2}}{2}c_{256} e^{247}+c_{256} e^{256}+\left(c_{167}+\frac{\sqrt{6}}{3} c_{347}\right) e^{257}\\
&+\left(\frac{\sqrt{6}}{2}
   c_{156}+\sqrt{2} c_{237}\right) e^{345}+\frac{\sqrt{6}}{2}c_{256} e^{346}+c_{347} e^{347}+\sqrt{2} c_{347} e^{356}+c_{357} e^{357}.
\end{aligned}
\end{equation*}
Therefore, 
$
\alpha_1\wedge \beta_1=0
$
describes a system of 6 equations. Hence, after apply the result of \cite [Proposition 4.5]{Schulte} for $X=e_2, \dots, e_7$,
we obtain a system of 42 equations. This system together with the fact that $c_{167} \neq 0$ and the condition $g_{\phi} = \langle \cdot, \cdot \rangle_{\mathfrak{n}_7}$  imply that any closed $G_2$ form on $\mathfrak{n}_7$ satisfying \eqref{productorio} is expressed as follows
\begin{equation}
\begin{aligned}\label{phin7}
\phi=&c_{123} e^{123}+c_{145} e^{145}+c_{167} e^{167}+c_{246} e^{246}+\left(c_{167}+\frac{\sqrt{6}}{3} c_{347}\right)e^{257}\\
&+c_{347} e^{347}+\sqrt{2} c_{347} e^{356}.
\end{aligned}
\end{equation}
Now we denote by $G_{\phi}$ the matrix associated to the inner product on $\mathfrak{n}_7$ induced by the 3-form $\phi$ given by \eqref{phin7}. Then, the condition $g_{\phi} = \langle \cdot, \cdot \rangle_{\mathfrak{n}_7}$ implies  \eqref{sistema2}
 is satisfied for some coefficients $c_{ijk}$. From equations $S_{11}= S_{33}= S_{44}=S_{66}=0$ we have 
\begin{equation*}
c_{123}=\frac{\sqrt{2}}{2c_{347}^3}, \quad c_{145}=-\sqrt{2}c_{347}, \quad c_{167}=-c_{347}, \quad \text{and} \quad c_{246}=\frac{\sqrt{2}}{2c_{347}^3}.
\end{equation*} 
Therefore $S_{55}=1$ and so $S\neq 0$ which is a contradiction with \eqref{sistema2}.

\medskip
Let $\mathfrak{n}_8$ be the Lie algebra described by the structure equations
$$\mathfrak{n}_8=(0,0,e^{12},-e^{13},-e^{23},e^{15}+e^{24},-e^{16}-e^{34}),$$ 
and let $\langle\cdot,\cdot\rangle_{\mathfrak{n}_8}$ be the inner product on $\mathfrak{n}_8$ such that $\{e^1,\dots , e^7\}$ is orthonormal. Then, $\langle\cdot,\cdot\rangle_{\mathfrak{n}_8}
=\sum_{i=1}^7  (e^i)^2$ is a nilsoliton because its Ricci tensor 
$$Ric=diag\Big(-2,-\frac{3}{2},-1,-\frac{1}{2},0,\frac{1}{2},1\Big)$$
satisfies \eqref{nilsoliton:lie-algebra}, for $\lambda=-\frac{5}{2}$ and  
$$D= diag \Big(\frac{1}{2},1,\frac{3}{2},2,\frac{5}{2},3,\frac{7}{2}\Big). 
$$
The nilsoliton inner product is unique (up to isometry and scaling) therefore it suffices to prove that there is no closed $G_2$ form on $\mathfrak{n}_8$ inducing such an inner product. A generic closed 3-form $\gamma$ on $\mathfrak{n}_8$ has the following expression
\begin{equation*}
\begin{aligned}
\gamma=&c_{123} e^{123}+c_{124} e^{124}+c_{125} e^{125}+c_{126} e^{126}+c_{1,2,7} e^{127}+c_{134} e^{134}+c_{135} e^{135}\\
&+c_{136} e^{136}+c_{137}
   e^{137}+\left(-c_{127}-c_{136}\right) e^{145}+c_{146} e^{146}+c_{147} e^{147}+c_{156} e^{156}\\
   &+c_{157} e^{157}+c_{234} e^{234}+c_{235}
   e^{235}+c_{2,3,6} e^{236}+c_{237} e^{237}+c_{236} e^{245}\\
   &+\left(c_{156}-c_{237}\right) e^{246}+c_{157} e^{247}+c_{256} e^{256}+c_{267}
   e^{267}+\left(c_{237}-2 c_{156}\right) e^{345}\\
   &+c_{157} e^{346}-c_{267} e^{357}+c_{267} e^{456},
\end{aligned}
\end{equation*}
where $c_{ijk}$ are real numbers. Now, we show conditions on the coefficients $c_{ijk}$ so that $\phi=\gamma$ is a closed $G_2$ form such that  $g_{\phi} = \langle \cdot, \cdot \rangle_{\mathfrak{n}_8}$. We apply the result previously mentioned \cite [Proposition 4.5]{Schulte} for $X=e_i$  $(1 \leq i \leq 7)$ and so $\eta= e^i$ by the condition $g_{\phi} = \langle\cdot,\cdot\rangle_{\mathfrak{n}_8}$. After solving the system of 42 equations we have that any closed $G_2$ form on $\mathfrak{n}_8$ satisfying \eqref{productorio} is expressed as follows
\begin{equation}\label{phin8}
\begin{aligned}
\phi=&c_{123} e^{123}+c_{124} e^{124}+c_{125} e^{125}+c_{126} e^{126}+c_{135} e^{135}+c_{136} e^{136}-c_{136} e^{145}\\
&+c_{234} e^{234}+c_{235}e^{235}+c_{236} e^{236}+c_{236} e^{245}+c_{256} e^{256}.
\end{aligned}
\end{equation}
Now denote by $G_{\phi}$ the matrix associated to the inner product on $\mathfrak{n}_8$ induced by the 3-form $\phi$ given by \eqref{phin8}. 
Then  $G_{\phi}=0$  obtaining a contradiction with \eqref{sistema2}.

\medskip

It only remains to study the Lie algebra $\mathfrak{n}_{11}$. 
According to Theorem \ref{classification}, 
$\mathfrak{n}_{11}$ is defined by the equations
$$\mathfrak{n}_{11}=(0,0,f^{12},0,f^{13},f^{24}+f^{23},f^{25}+f^{34}+f^{15}+f^{16}-3f^{26}).$$
We consider the new basis $\{e^j\}_{j=1}^7$ of $\mathfrak{n}_{11}^*$ with
\begin{align*}
\{&e^1=f^2,e^2=-\frac{\sqrt{3}}{3}f^1, e^3=\frac{\sqrt{39}}{39}f^3+\frac{\sqrt{39}}{78}f^4,e^4=-\frac{\sqrt{78}}{78}f^4,\\
& e^5=\frac{\sqrt{3}}{39}f^6,e^6=-\frac{1}{3}f^5,e^7=-\frac{\sqrt{3}}{1014}f^7\}.
\end{align*}
Thus, the Lie algebra $\mathfrak{n}_{11}$ can also be described by the structure equations
\begin{align*}
\mathfrak{n}_{11}=&\Big(0,0,\frac{\sqrt{13}}{13}e^{12},0,\frac{\sqrt{13}}{13}e^{13}-\frac{\sqrt{26}}{26}e^{14},\frac{\sqrt{26}}{26}e^{24}+\frac{\sqrt{13}}{13}e^{23},\\
&\frac{\sqrt{13}}{26}e^{25}+\frac{\sqrt{26}}{26}e^{34}+\frac{\sqrt{39}}{26}e^{15}+\frac{\sqrt{13}}{26}e^{16}-\frac{\sqrt{39}}{26}e^{26}\Big),
\end{align*}
Let $\langle \cdot, \cdot \rangle_{\mathfrak{n}_{11}}$ be the inner product on $\mathfrak{n}_{11}$ such that $\{e^1, \dots, e^7 \}$ is orthonormal. Then, $\langle \cdot, \cdot \rangle_{\mathfrak{n}_{11}} = \sum_{i = 1}^7 (e^i)^2$ is a nilsoliton because its Ricci tensor
$$Ric=\frac{1}{52}diag(-7,-7,-3,-3,1,1,5)$$
satisfies $Ric=-\frac{11}{52}Id+D$, where $D$ is the derivation of the Lie algebra $\mathfrak{n}_{11}$ given by
$$D=\frac{1}{13}diag(1,1,2,2,3,3,4).$$
It suffices to prove that there is no closed $G_2$ form on $\mathfrak{n}_{11}$ inducing such an inner product. Let's suppose that $\mathfrak{n}_{11}$ has a closed $G_2$ form $\phi$ such that  $g_{\phi} = \langle\cdot,\cdot\rangle_{\mathfrak{n}_{11}}$. A generic closed 3-form $\gamma$ on $\mathfrak{n}_{11}$ has the following expression
\begin{equation*}
\begin{aligned}
\gamma=&c_{123} e^{123}+c_{124} e^{124}+c_{125} e^{125}+c_{126} e^{126}+c_{127} e^{127}+c_{134} e^{134}+c_{135} e^{135}\\
&+c_{136} e^{136}+c_{137}
   e^{137}+c_{145} e^{145}+c_{146} e^{146}+c_{147} e^{147}+c_{156} e^{156}-\sqrt{\frac{3}{2}} c_{347} e^{157}\\
   &+\frac{c_{347}
   e^{167}}{\sqrt{2}}+c_{234} e^{234}+c_{235} e^{235}+c_{236} e^{236}+\left(\frac{c_{137}}{\sqrt{3}}-\frac{2 c_{156}}{\sqrt{3}}\right)
   e^{237}\\
   &+\left(\frac{c_{127}}{\sqrt{2}}-\frac{c_{136}}{\sqrt{2}}+c_{146}-\frac{c_{235}}{\sqrt{2}}\right) e^{245}+c_{246} e^{246}+c_{247}
   e^{247}+\left(\frac{c_{156}}{\sqrt{3}}-\frac{2 c_{137}}{\sqrt{3}}\right) e^{256}\\
   &+\frac{c_{347} e^{257}}{\sqrt{2}}+\sqrt{\frac{3}{2}} c_{347}
   e^{267}+\left(\frac{1}{2} c_{147}-\frac{c_{156}}{\sqrt{2}}-\frac{1}{2} \sqrt{3} c_{247}\right) e^{345}\\
   &+\left(-\sqrt{\frac{2}{3}}
   c_{137}-\frac{1}{2} \sqrt{3} c_{147}+\frac{c_{156}}{\sqrt{6}}-\frac{1}{2} c_{247}\right) e^{346}+c_{347} e^{347}+\sqrt{2} c_{347} e^{356}
\end{aligned}
\end{equation*}
where $c_{ijk}$ are arbitrary real numbers.

Now, we show conditions on the coefficients $c_{ijk}$ so that $\phi=\gamma$ is a closed $G_2$ form such that  $g_{\phi} = \langle\cdot,\cdot\rangle_{\mathfrak{n}_{11}}$. We apply the result of \cite [Proposition 4.5]{Schulte} for $X=e_i$  $(1 \leq i \leq 7)$ and so $\eta= e^i$ by the condition $g_{\phi} = \langle\cdot,\cdot\rangle_{\mathfrak{n}_{11}}$. After solving the system of 42 equations we have that any closed $G_2$ form on $\mathfrak{n}_{11}$ satisfying \eqref{productorio} is expressed as follows
\begin{equation}\label{phin11}
\begin{aligned}
\phi=&c_{123} e^{123}-c_{246} e^{145}-\sqrt{3}c_{246} e^{167}-\frac{\sqrt{6}}{2}c_{347} e^{157}+\frac{\sqrt{2}}{2}c_{347} e^{167}-\sqrt{3} c_{246}e^{245}
\\&+c_{246} e^{246}+\frac{\sqrt{2}}{2}c_{347} e^{257}+\frac{\sqrt{6}}{2}c_{347} e^{267}+ c_{347} e^{347}+\sqrt{2}c_{347}e^{356}.
\end{aligned}
\end{equation}
As before denote by $G_{\phi}$ the matrix associated to the inner product  induced by the 3-form $\phi$ given by \eqref{phin11}. Then, the condition $g_{\phi} =  \langle\cdot,\cdot\rangle_{\mathfrak{n}_{11}}$ implies  \eqref{sistema2}, for some $c_{ijk}$.
Equations $S_{66}=S_{77}=0$ imply that 
$$c_{246}=-\frac12c_{347}, \hspace{0.2cm} \text{and} \hspace{0.2cm} c_{347}=2^{-1/3}.$$ 
Therefore, $S_{44}=-\frac{1}{2}$ and so $S \neq 0$ which contradicts \eqref{sistema2}.
\end{proof}

 \begin{remark}  Note that the $4$-step nilpotent Lie algebra $\frak n_{10}$  is isomorphic  in the classification given in \cite{Fe3}  to the Lie algebra  $1.3(i)[ \lambda = 1]$ and  the existence of the nilsoliton  was shown in \cite[Example 2]{Fe1}. Since  an explicit expression of the nilsoliton  is not known, we  cannot apply  the argument used in the proof of  Proposition \ref{nilsoliton-NO-closed G2}. Thus, it remains open the question of whether the Lie algebra $\frak n_{10}$ admits a
closed $G_2$ form inducing a nilsoliton or not.
  Moreover, the explicit expression of the nilsolitons for ${\frak n}_{11}$ and ${\frak n}_{12}$ have been already determined in \cite{Fe3} (see there page 20, Remark 3.5), but our basis
 is different for the nilsoliton on the other Lie algebras.\end{remark}

\begin{theorem}\label{nilsoliton-closed G2}
Up to isomorphism, $\mathfrak{n}_2$, $\mathfrak{n}_4$, $\mathfrak{n}_6$ and $\mathfrak{n}_{12}$
are the unique $s$-step nilpotent Lie algebras $(s=2, 3)$ with a nilsoliton inner product
determined by a closed $G_2$-structure.
\end{theorem}
\begin{proof}
We will show that the Lie algebra $\mathfrak{n}_i$ $(i=2,4,6,12)$ has
a closed $G_2$ form $\varphi_i$ 
such that the Ricci tensor of the inner product $g_{\varphi_i}$
satisfies \eqref{nilsoliton:lie-algebra}, for some derivation 
$D$ of $\mathfrak{n}_i$ and some real number $\lambda$.

For $\mathfrak{n}_2$ we consider 
the closed $G_2$ form $\varphi_2$ defined by
\begin{equation}\label{def: varphi_2}
\varphi_2 =e^{147}+e^{267}+e^{357}+e^{123}+e^{156}+e^{245}-e^{346}.
\end{equation}
The inner product $g_{\varphi_2}$ given by \eqref{metric} is the one making 
orthonormal the basis $\{e^1,\dotsc, e^7\}$, and it is a nilsoliton since 
$
Ric =-2I_7+D,$  where 
$$
D = diag \left (1, \frac 32, \frac 32, 2, \frac{5}{2}, \frac{5}{2}, 2  \right )
$$
is a derivation of $\mathfrak{n}_2$.

\bigskip

On the Lie algebra $\mathfrak{n}_4$, we define the $G_2$ form $\varphi_4$ by
\begin{equation}\label{def: varphi_4}
\varphi_4 =-e^{124}-e^{456}+e^{347}+e^{135}+e^{167}+e^{257}-e^{236}.
\end{equation}
Then, $\varphi_4$ is closed, the inner product $g_{\varphi_4}$
makes the basis $\{e^1,\dotsc, e^7\}$ orthonormal and $g_{\varphi_4}$  is a nilsoliton 
since $Ric =- \frac 5 2 I_7+D$, where $D$ is the derivation of $\frak n_4$ given by
$$
D = diag \left (1, \frac 32, \frac 52, 2, 2, \frac 72, 3 \right).
$$

\bigskip

For the Lie algebra $\mathfrak{n}_6$ we consider the closed $G_2$-structure defined by the $3$-form
\begin{equation}\label{def: varphi_6}
\varphi_6 =e^{123}+ e^{145}+e^{167} + e^{257} - e^{246} +e^{347}+e^{356}.
\end{equation}
Therefore, the inner product $g_{\varphi_6}$ is such that the basis $\{e^1,\dotsc, e^7\}$ is orthonormal 
and it is a nilsoliton since $Ric =- \frac 52  I_7+D$, where $D$ is the derivation of $\frak n_6$ given by
$$
D = diag \left ( \frac 12, 2, 2, \frac 52, \frac 52, 3, 3 \right ).
$$

\bigskip
Theorem \ref{classification} implies that the Lie algebra $\mathfrak{n}_{12}$
is defined by the equations
$$\mathfrak{n}_{12}=(0,0,0,h^{12},h^{23},-h^{13},2h^{26}-2h^{34}-2h^{16}+2h^{25}).$$
We consider the basis $\{e^i\}_{i=1}^7$ of $\mathfrak{n}_{12}^*$ given by
\begin{align*}
\{&e^1=\frac{\sqrt{3}}{2}h^2, e^2=h^1-\frac12h^2, e^3=h^3, e^4=-\frac14 h^4, e^5=\frac14 h^5+\frac14h^6,\\
& e^6=-\frac{\sqrt{3}}{12}h^5+\frac{\sqrt{3}}{12}h^6, e^7=-\frac{\sqrt{3}}{48}h^7\}.
\end{align*}
Then, $\mathfrak{n}_{12}$ is defined as follows
\begin{equation}\label{n12}
\begin{aligned}
\mathfrak{n}_{12}=&\Big(0,0,0,\frac{\sqrt{3}}{6}e^{12}, -\frac{1}{4}e^{23}+\frac{\sqrt{3}}{12}e^{13}, -\frac{\sqrt{3}}{12}e^{23}-\frac{1}{4}e^{13},\\
&-\frac{\sqrt{3}}{6}e^{34}+\frac{\sqrt{3}}{12}e^{25}+\frac{1}{4}e^{26}+\frac{\sqrt{3}}{12}e^{16}-\frac{1}{4}e^{15} \Big).
\end{aligned}
\end{equation}
We define the $G_2$ form $\varphi_{12}$ by
\begin{equation}\label{varphi12}
\varphi_{12}=-e^{124}+e^{135}+e^{167}-e^{236}+e^{257}+e^{347}-e^{456}.
\end{equation}
Clearly $\varphi_{12}$ is closed. Moreover, $\varphi_{12}$ defines the inner product $g_{\varphi_{12}}$ which 
makes the basis $\{e^1, \dots, e^7\}$ orthonormal, and $g_{\varphi_{12}}$ is a nilsoliton since 
$Ric=-\frac14Id+\frac18D$, where $D$ is the derivation of $\mathfrak{n}_{12}$ given by
$$D=diag(1,1,1,2,2,2,3).$$
\end{proof}
\end{section}


\begin{section}{Laplacian flow}\label{sectLaplacian}
Let us consider the nilpotent Lie algebra $\mathfrak{n}_i$
$(i=2,4,6)$ defined in Theorem \ref{classification}, and the Lie algebra $\mathfrak{n}_{12}$ defined by 
 \eqref{n12}. Let $N_i$ be the
simply connected nilpotent Lie group 
with Lie algebra $\mathfrak{n}_i$, and let
$\varphi_i$ be the closed $G_2$ form on $N_i$ $(i=2,4,6,12)$
given by \eqref{def: varphi_2}, 
\eqref{def: varphi_4}, \eqref{def: varphi_6} and \eqref{varphi12}, for $i=2, 4,6$ and $12$, respectively.  

The purpose of this section is to prove long time existence and
uniqueness of solution for the Laplacian flow of $\varphi_i$ on $N_i$,
and  that   the underlying metrics  $g(t)$ of this solution converge smoothly, up to pull-back by time-dependent diffeomorphisms,  to a flat metric, uniformly on compact sets in $N_i$, as $t$ goes to infinity.

Let $M$ be a $7$-dimensional manifold with an arbitrary $G_2$ form $\varphi$.
The {\em Laplacian flow} of $\varphi$ is defined to be
\begin{equation*}\left\{
\begin{aligned}
&\frac{d}{dt} \varphi(t)=\Delta_{t} \varphi(t), \\
&\varphi(0)=\varphi,
\end{aligned}
\right.
\end{equation*}
where $\Delta_t$ is the Hodge Laplacian of the 
metric $g_t$ determined by the $G_2$ form $\varphi(t)$.

For the different types of $G_2$-structures the behavior of the solution of the 
Laplacian flow is very different. For example, the stable solutions of the Laplacian 
flow are given by the $G_2$ manifolds $(M, \varphi)$ such that $Hol(M)\subseteq G_2$.

The study of the Laplacian flow of a 
closed $G_2$ form $\varphi$ on a manifold $M$
consists to study  long time existence, convergence and formation of singularities  for the system
of differential equations
\begin{equation} \label{eq-Lap-flow:2}\left\{
\begin{aligned}
&\frac{d}{dt} \varphi(t)=\Delta_{t} \varphi(t), \\
&d\varphi(t)=0,\\
&\varphi(0)=\varphi.
\end{aligned}
\right.
\end{equation}

In the case of closed $G_2$-structures on compact manifolds, Bryant and Xu \cite{Br-Xu}
gave a result of short time existence and uniqueness of solution.

\begin{theorem}\cite{Br-Xu}
If $M$ is compact, then \eqref{eq-Lap-flow:2}
has a unique solution for a short time $0\leq t < \epsilon$, with $\epsilon$ 
depending on $\varphi=\varphi(0)$.
\end{theorem}

In the following theorem we determine a global solution of
the Laplacian flow of the closed $G_2$ form $\varphi_2$ on $N_2$.

\begin{theorem}\label{Lap-flow: N2}
The family of closed $G_2$ forms $\varphi_2(t)$ on $N_2$ given by
\begin{equation}\label{solution:N2}
\varphi_2(t)=e^{147}+e^{267}+e^{357}+f(t)^{3}e^{123}+e^{156}+e^{245}-e^{346}, \qquad t\in \left (-\frac{3}{10},+ \infty \right),
\end{equation}
is the solution of the Laplacian flow \eqref{eq-Lap-flow:2} of $\varphi_2$, where $f=f(t)$ is the function
$$
f(t)=\Big(\frac{10}{3} t +1\Big)^{\frac{1}{5}}.
$$
Moreover, 
the underlying metrics  $g(t)$ of this solution converge smoothly, up to pull-back by time-dependent diffeomorphisms,  to a flat metric, uniformly on compact sets in $N_2$, as $t$ goes to infinity.
\end{theorem}
\begin{proof}
Let $f_i=f_i(t)$ $(i=1,\dots,7)$ be some differentiable real functions depending on a parameter 
$t\in I\subset{\mathbb{R}}$ such that $f_i(0)=1$ and $f_i(t)\not=0$, for any $t\in I$, where $I$ is a real open interval. 
For each $t\in I$, we consider the basis $\{x^1,\dots,x^7\}$ of left invariant $1$-forms
on $N_2$ defined by
$$
x^i=x^i(t)=f_{i}(t)e^i, \quad 1\leq i\leq 7.
$$
From now on we write $f_{ij}=f_{ij}(t)=f_{i}(t)f_{j}(t)$, $f_{ijk}=f_{ijk}(t)=f_{i}(t)f_{j}(t)f_{k}(t)$, and so forth.
Then, the structure equations of $N_2$ with respect to this basis are
\begin{equation} \label{new-eq:N2}
dx^i=0, \quad i=1,2,3,4,7,  \qquad dx^5=\frac{f_{5}}{f_{12}} x^{12},   \qquad dx^6=\frac{f_{6}}{f_{13}} x^{13}.
\end{equation}
Now, for any $t\in I$, we consider the $G_2$ form $\varphi_{2}(t)$ on $N_2$ given by
\begin{equation}\label{sol-N2-1}
\begin{aligned}
\varphi_{2}(t)&=x^{147}+x^{267}+x^{357}+x^{123}+x^{156}+x^{245}-x^{346} \\
&=f_{147}e^{147}+f_{267}e^{267}+f_{357}e^{357}+f_{123}e^{123}+f_{156}e^{156}
+f_{245}e^{245}-f_{346}e^{346}.
\end{aligned}
\end{equation}

Note that $\varphi_{2}(0)=\varphi_{2}$ and, for any $t$, the $3$-form $\varphi_{2}(t)$ 
on $N_2$ determines the metric $g_{t}$  
such that the basis $\{x_{i}={\frac{1}{f_i}}e_{i}; \ i=1,\dots,7\}$ of $\mathfrak{n}_2$ is orthonormal.
So, $ g (t) (e_i,e_i)={f_i}^2$. 

Using \eqref{new-eq:N2}, one can check that $d\varphi_{2}(t)=0$ if and only if
\begin{equation} \label{eq-Lap-flow:4}
f_{26}(t)=f_{35}(t),
\end{equation}
for any $t$. Assuming $f_{i}(0)=1$ and \eqref{eq-Lap-flow:4},
to solve the flow \eqref{eq-Lap-flow:2} of $\varphi_2$, 
we need to determine the functions $f_i$ and the interval $I$ so 
that $\frac{d}{dt}\varphi_{2}(t)=\Delta_{t}\varphi_{2}(t)$, for $t\in I$.
Using \eqref{sol-N2-1} we have
\begin{equation}\label{eq-Lap-flow:5}
\begin{aligned} 
\frac{d}{dt}\varphi_{2}(t)=&\Big(f_{147}\Big)' e^{147}+\Big(f_{267}\Big)' e^{267}+\Big(f_{357}\Big)' e^{357}+\Big(f_{123}\Big)' e^{123} \\
&+\Big(f_{156}\Big)' e^{156}+\Big(f_{245}\Big)' e^{245}-\Big(f_{346}\Big)' e^{346}.
\end{aligned}
\end{equation}
Now, we calculate $\Delta_{t}\varphi_{2}(t)=-d *_{t}d *_{t}\varphi_{2}(t)$. On the one hand, we have
\begin{equation} \label{eq-ast-N2}
\begin{aligned} 
*_{t}\varphi_{2}(t)=x^{2356}-x^{1345}-x^{1246}+x^{4567}+x^{2347}-x^{1367}+x^{1257}.
\end{aligned}
\end{equation}
So, $x^{4567}$ is the unique nonclosed summand in $*_{t}\varphi_{2}(t)$. Then, taking into account
\eqref{eq-Lap-flow:4}, we obtain
$$
d(*_{t}d*_{t}\varphi_{2}(t))={\frac{f_6}{f_{13}}} \left (-{\frac{f_6}{f_{13}}}x^{123}-{\frac{f_5}{f_{12}}}x^{123} \right )=-2\Big({\frac{f_6}{f_{13}}}\Big)^2 x^{123}.
$$
Therefore, in terms of the forms $e^{ijk}$, the expression of $-d(*_{t}d*_{t}\varphi_{2}(t))$ is
\begin{equation} \label{eq-Lap-flow:6}
-d(*_{t}d*_{t}\varphi_{2}(t))=2f_{123}\Big({\frac{f_6}{f_{13}}}\Big)^2 e^{123}=2\Big({\frac{f_{2}(f_6)^2}{f_{13}}}\Big)e^{123}.
\end{equation}
Comparing \eqref{eq-Lap-flow:5} and \eqref{eq-Lap-flow:6} we see that, in particular, 
$
f_{156}(t)=1,
$
for any $t\in I$. 
Then, using \eqref{eq-Lap-flow:4}, we have
$$
\frac{f_{2}(f_6)^2}{f_{13}}
=\frac{1}{(f_1)^2}.
$$
This equality and \eqref{eq-Lap-flow:6} imply that $-d(*_{t}d*_{t}\varphi_{2}(t))$ can be expressed as follows
\begin{equation} \label{eq-Lap-flow:7}
-d(*_{t}d*_{t}\varphi_{2}(t))=2\frac{1}{(f_1)^2}e^{123}.
\end{equation} 
Then, from \eqref{eq-Lap-flow:5} and \eqref{eq-Lap-flow:7} we have 
that $\frac{d}{dt}\varphi_{2}(t)=\Delta_{t}\varphi_{2}(t)$ if and only if the functions $f_{i}(t)$ satisfy 
the following system of differential equations
\begin{equation} \label{eq-Lap-flow:8}
\begin{aligned} 
&\Big(f_{147}\Big)'=\Big(f_{267}\Big)'=\Big(f_{357}\Big)'=\Big(f_{156}\Big)'=\Big(f_{245}\Big)'=\Big(f_{346}\Big)'=0, \\
&\Big(f_{123}\Big)'=2\frac{1}{(f_1)^2}.
\end{aligned} 
\end{equation}
Because $\varphi_{2}(0)=\varphi_2$, the equations in the first line
of \eqref{eq-Lap-flow:8} imply
\begin{equation} \label{eq-Lap-flow:9}
f_{147}(t)=f_{267}(t)=f_{357}(t)=f_{156}(t)=f_{245}(t)=f_{346}(t)=1,
\end{equation}
for any $t\in I$. 
From the equations \eqref{eq-Lap-flow:9} 
we obtain
$$
f_{1}^2=f_{2}^2=f_{3}^2.
$$
Let us consider $f=f_1=f_2=f_3$. 
Using again \eqref{eq-Lap-flow:9}
we have
$$
f_{i}(t)=\Big(f(t)\Big)^{-\frac{1}{2}}, \quad i=4,5,6,7.
$$
Now, the last equation of \eqref{eq-Lap-flow:8} implies that
$
f^4 f'=\frac{2}{3}.
$
Integrating this equation, we obtain
$$
f^5={\frac{10}{3}}t+B, \quad B=constant.
$$
But $\varphi_{2}(0)=\varphi_2$ implies $f^3(0)=f_{123}(0)=1$, that is, $B=1$. Hence,
$$
f(t)=\Big({\frac{10}{3}}t+1\Big)^{\frac{1}{5}},
$$
and so the one parameter family of $3$-forms $\{\varphi_{2}(t)\}$ given
by \eqref{solution:N2} is the solution of the Laplacian flow of 
$\varphi_{2}$ on $N_2$, and it is defined for every $t \in (- \frac{3}{10}, + \infty)$.

To complete the proof, we study the behavior of  the underlying metric $g(t)$  of such a solution in the limit for $t \to + \infty$.
Indeed, if we think of the Laplacian flow as a one parameter family of $G_2$ manifolds with a closed
$G_2$-structure,
it can  be checked that, in the limit, the resulting manifold has vanishing curvature. For every $t\in \left (-\frac{3}{10},+ \infty \right)$, denote by
$g(t)$  the metric on $N_{2}$ induced by the $G_2$ form $\varphi_{2}(t)$ given by \eqref{solution:N2}. Then,
\begin{equation*}
\begin{aligned}
g(t)&=\Big(\frac{10}{3}t+1\Big)^{2/5} (e^1)^2 +\Big(\frac{10}{3}t+1\Big)^{2/5} (e^2)^2 +\Big(\frac{10}{3}t+1\Big)^{2/5} (e^3)^2
\\&+\Big(\frac{10}{3}t+1\Big)^{-1/5} (e^4)^2 +\Big(\frac{10}{3}t+1\Big)^{-1/5} (e^5)^2 +\Big(\frac{10}{3}t+1\Big)^{-1/5} (e^6)^2
\\&+\Big(\frac{10}{3}t+1\Big)^{-1/5} (e^7)^2.
\end{aligned}
\end{equation*}
Concretely, taking into account the symmetry properties of the Riemannian curvature  $R(t)$ we obtain
\begin{align*}
R_{1212}&=R_{1313}=-\frac{3}{4(1+\frac{10}{3}t)},\\
R_{1515}&=R_{1616}=R_{3636}=R_{2525}=\frac{1}{4(1+\frac{10}{3}t)},\\
R_{2356}&=-\frac{1}{4(1+\frac{10}{3}t)}, \quad R_{ijkl} =0 \quad \text{otherwise},
\end{align*}
where $R_{ijkl}=R(t) (e_i,e_j,e_k,e_l)$. Therefore, $\lim_{t\rightarrow +\infty}R (t)=0$.
\end{proof}

\begin{remark}
Note that, for every $t \in \left (-\frac{3}{10},+ \infty \right)$, the metric $g(t)$ is a nilsoliton on the Lie algebra $\frak n_{2}$ of $N_{2}$.
In fact,  with respect to the orthonormal basis $(x_1(t), \ldots, x_7 (t))$, we have \begin{equation*} Ric(g(t)) = - \frac{6}{(3+10t)} Id +  \frac{3}{(3+10t)} diag  \left (1, \frac{3}{2},\frac{3}{2}, 2, \frac{5}{2}, \frac{5}{2}, 2 \right)=\frac{3}{(3+10t)}Ric(g(0))\end{equation*} with  $\frac{3}{(3+10t)} diag  \left (1, \frac{3}{2}, \frac{3}{2}, 2, \frac{5}{2}, \frac{5}{2}, 2 \right)$ a derivation of $\frak n_2$ for every $t$.
\end{remark}

\begin{remark}
The limit  can be also computed  fixing  the  $G_2$-structure  and  changing  the  Lie  bracket  as in \cite{La4}.  We  evolve  the Lie brackets $\mu(t)$ 
instead of the   $3$-form defining  the $G_2$-structure and  we can show that the corresponding bracket flow has a solution for every $t$.  Indeed,  if we fix  on $\R^7$ the $3$-form $x^{147}+x^{267}+x^{357}+x^{123}+x^{156}+x^{245}-x^{346}$,    the basis $(x_1(t), \ldots, x_7 (t))$      defines  for every positive  $t$ a  nilpotent Lie algebra with bracket $\mu(t)$ such that $\mu(0)$ is the Lie bracket of $\frak n_2$. Moreover, the solution converges to the null bracket corresponding to the abelian Lie algebra.  \end{remark}

In order to prove long time existence of solution for the Laplacian 
flow \eqref{eq-Lap-flow:2} of the closed $G_2$ form $\varphi_4$ on $N_4$, we need 
to study the (nonlinear) system of ordinary differential equations
\begin{equation}\label{eqn:1-N4}
\begin{cases}
u'=+\dfrac 23\,\dfrac{2-u^3}{u^3v^3},\\ v'= - \dfrac 23\,\dfrac{1-2u^3}{u^4v^2},\end{cases}
\end{equation} 
with initial conditions
\begin{equation}\label{eqn:2-N4}
u(0)=v(0)=1,
\end{equation}
where $u=u(t)$ and $v=v(t)$ are differentiable real functions such that are both positive.
Note that 
the first equation of \eqref{eqn:1-N4} implies that $u'>0$ since $u(0)=1$, 
$u=u(t)>0$ and $v=v(t) >0$. 
Moreover, we note also that the functions at the second member of \eqref{eqn:1-N4} are $C^\infty$ in 
the domain
\begin{equation*}
\Omega=\bigl\{(u,v)\in\mathbb R^2\mid 0<u<2^{1/3},\, v>0\bigr\},
\end{equation*}
in the phase plane. Then, for every point $(u_0,v_0)\in \Omega$, there exists a unique maximal 
solution $(u,v)$, which has $(u_0,v_0)$ as initial condition, 
and with existence domain a certain open interval $I$ such that either 
$$
\lim_{t\to \inf I}\left(u(t)^2+v(t)^2\right)=+\infty,
$$
or
$$
\lim_{t\to \inf I}\bigl(u(t),v(t)\bigr)\in \partial \Omega,
$$
and either 
$$
\lim_{t\to \sup I}\left(u(t)^2+v(t)^2\right)=+\infty,
$$
or 
$$
\lim_{t\to \sup I}\bigl(u(t),v(t)\bigr)\in \partial \Omega,
$$
where $\partial \Omega$ denotes the boundary of $\Omega$.

\begin{proposition}\label{system:N4}
The maximal solution $\bigl(u(t),v(t)\bigr)$ of \eqref{eqn:1-N4}, satisfying the initial 
conditions \eqref{eqn:2-N4}, belongs to the trajectory of equation 
\begin{equation}\label{eqn:3-N4}
v=\frac 1{\sqrt{u(2-u^3)}}.
\end{equation}
\end{proposition}
\begin{proof}
From \eqref{eqn:1-N4} we obtain
\begin{equation*}
\frac {dv}{du}=-\frac {v(1-2u^3)}{u(2-u^3)},
\end{equation*}
that is,
\begin{equation*}
\frac {dv}{v}=-\frac {1-2u^3}{u(2-u^3)}\,du.
\end{equation*}
Integrating this equation and using \eqref{eqn:2-N4},
we have 
\begin{equation*}
\log v=\log \bigl(u(2-u^3)^{-1/2}\bigr).
\end{equation*}
Therefore,
\begin{equation*}
v=\frac 1{\sqrt{u(2-u^3)}}.\qedhere
\end{equation*}
\end{proof}
As a consequence we have the following corollary.
\begin{corollary}\label{interval:N4}
The maximal solution of \eqref{eqn:1-N4}-\eqref{eqn:2-N4},
\begin{equation*}
I\ni t\mapsto \bigl(u(t),v(t)\bigr)\in\Omega
\end{equation*}
parametrizes the whole curve \eqref{eqn:3-N4}. 
Moreover,
the maximal solution is defined in the interval
$$
I=(t_{min}, +\infty),
$$
where
\begin{equation}\label{exprtmin}
t_{min} = - \frac{3}{2}  \int_{0}^1 \frac{x^{3/2}}{(2 - x^3)^{5/2}} dx,
\end{equation}
and 
\begin{equation*}
\begin{cases}\lim_{t\to t_{min}} u(t)=0,\\ \lim_{t\to t_{min}}v(t)=+\infty,\end{cases}
\qquad\begin{cases}\lim_{t\to + \infty} u(t)=2^{1/3},\\ \lim_{t\to + \infty}v(t)=+\infty.\end{cases}
\end{equation*}
\end{corollary}
\begin{proof} Let $I= (t_{min}, t_{max})$ the existence interval
of the maximal solution $(u(t), v(t))$ of \eqref{eqn:1-N4} satisfying the initial 
conditions \eqref{eqn:2-N4}. Using the previous proposition and the first 
equation of \eqref{eqn:1-N4} we see that
$$
v(t)=( 2 u(t) -u(t)^4)^{-1/2},  \qquad  u'(t)= - \frac{2 u(t)^3 - 4}{3 u(t)^3 v(t)^3},
$$
which imply
$$
u'(t)=\frac{2}{3} \frac {(2 - u(t)^3)^{ \frac 5 2 }} {u(t)^{\frac 3 2}}.
$$
We define the functions $x(t)$ and $f(x)$ by
$$
x(t)=u(t), \quad f(x)= \frac{2}{3} \frac {(2 - x^3)^{\frac 5 2 }} {x^{\frac 3 2 }}.
$$ 
In order to find $t_{max}$, we can use that $\frac{dx}{dt} = f(x(t))$ or, equivalently, 
$$
\frac{dx}{f(x)} = dt.
$$
So, in particular, we have $$\frac{dt}{dx}= \frac{3}{2}x^{ \frac 3 2}(2-x^3)^{- \frac 5 2}.$$
Note that the function $\frac{3}{2} x^{ \frac 3 2} (2-x^3)^{- \frac 5 2}$ is increasing
from $0$, for $x=0$, to $+ \infty$, for $x=2^{ \frac 13}$. 
Then, integrating 
$\frac{dx}{f(x)}=dt$ between $t_{min}$ and $0$,
and using that $x(t_{min})=0$ and $x(0) =1$, we have that $t_{min}$ is finite  
and equal to the real number
$$
t_{min}=- \frac 3 2  \int_{0}^1 x^{3/2} (2 - x^3)^{-5/2} dx.
$$
Similarly, in order to find $t_{max}$ we integrate again  
$\frac{dx}{f(x)}=dt$ between $0$ and 
$t_{max}$. Since $x(t_{max} )= 2^{\frac 1 3}$ we get
$$
t_{max}=-\frac 3 2  \int_{1}^{2^{\frac 1 3}} x^{\frac 3 2} (2 - x^3)^{- \frac 5  2} dx,
$$
which implies that $t_{max}$ is $+ \infty$ because this integral is not defined in 
$x =2^{\frac 1 3}$. 
\end{proof}

\begin{theorem}\label{Lap-flow: N4} 
There exists a solution $\varphi_{4}(t)$
of the Laplacian flow of $\varphi_4$ on $N_4$ defined in the interval
$I=(t_{min},+\infty)$, where $t_{min}$ is the negative real number given by the elliptic integral
$$
t_{min}=  -\frac{3}{2} \int_{0}^1 x^{3/2} (2 - x^3)^{-5/2} dx.
$$
Moreover, the  underlying metrics  $g(t)$ of this solution converge smoothly, up to pull-back by time-dependent diffeomorphisms,  to a flat metric, uniformly on compact sets in $N_4$, as $t$ goes to infinity.
\end{theorem}
\begin{proof}
Let us consider some differentiable real functions
$f_i=f_{i}(t)$ $(i=1,\dots,7)$ and $h_{j}=h_{j}(t)$ $(j=1,2,3)$ depending on a parameter 
$t\in I\subset{\mathbb{R}}$ such that $f_{i}(0)=1, h_{j}(0)=0$ and $f_{i}(t) \not=0$, for any $t\in I$
and for any $i$ and $j$.
For each $t\in I$, we consider the basis $\{x^1,\dots,x^7\}$ of left invariant $1$-forms
on $N_4$ defined by
\begin{equation*}
\begin{aligned}
x^i=x^i(t)=f_{i}(t)e^i,\quad 1\leq i\leq 4,  \quad x^5=x^{5}(t)=f_5(t)e^5+h_{1}(t)e^1, \\
x^6=x^{6}(t)=f_6(t)e^6+h_{2}(t)e^2,  \quad
x^7=x^{7}(t)=f_7(t)e^7+h_{3}(t)e^4. 
\end{aligned}
\end{equation*}
The structure equations of $N_4$ with respect to this basis are
\begin{equation} 
\begin{aligned}\label{new-eq:N4}
dx^i=0, \quad i=1,2,4,5,  \quad dx^3=\frac{f_{3}}{f_{12}} x^{12},   \\
dx^6=\frac{f_{6}}{f_{13}} x^{13}+\frac{f_{6}}{f_{24}} x^{24}, \quad dx^7=\frac{f_{7}}{f_{15}} x^{15}.
\end{aligned}
\end{equation}
For any $t\in I$, we define the $G_2$ form $\varphi_{4}(t)$ on $N_4$ by
\begin{equation}
\begin{aligned}\label{sol-N4-1}
\varphi_{4}(t)&=-x^{124}-x^{456}+x^{347}+x^{135}+x^{167}+x^{257}-x^{236} \\
&=\Big(-f_{124}-f_{4}h_{12}-f_{2}h_{13}+f_{1}h_{23}\Big)e^{124}-f_{456}e^{456}+f_{347}e^{347} \\
&+f_{135}e^{135}+f_{167}e^{167}+f_{257}e^{257}-f_{236}e^{236}+\Big(f_{46}h_1-f_{16}h_3\Big)e^{146}Ê\\
&-\Big(f_{45}h_2+f_{25}h_3\Big)e^{245}+\Big(-f_{27}h_1+f_{17}h_2\Big)e^{127}.
\end{aligned}
\end{equation}
Clearly $\varphi_{4}(0)=\varphi_{4}$ since $f_{i}(0)=1$ and $h_{j}(0)=0$. 
Moreover, using \eqref{new-eq:N4} and \eqref{sol-N4-1}, one can check that $d\varphi_{4}(t)=0$ if and only if
\begin{equation*} \label{eq-Lap-flow:N4}
f_{16}(t)=f_{34}(t), \quad f_{37}(t)=f_{56}(t),
\end{equation*}
for any $t$.

To study the flow \eqref{eq-Lap-flow:2} of $\varphi_4$, 
we need to determine the functions $f_i$, $h_j$ and the interval $I$ so 
that $\frac{d}{dt}\varphi_{4}(t)=\Delta_{t}\varphi_{4}(t)$, for $t\in I$.
On the one hand, using 
\eqref{sol-N4-1} we have 
\begin{equation}
\begin{aligned} \label{eq-Lap-flow:N4-5}
\frac{d}{dt}\varphi_{4}(t)&=\Big(-f_{124}-f_{4}h_{12}-f_{2}h_{13}+f_{1}h_{23}\Big)' e^{124}-\Big(f_{456}\Big)' e^{456}+\Big(f_{347}\Big)' e^{347} \\
&+\Big(f_{135}\Big)' e^{135}+\Big(f_{167}\Big)' e^{167}+\Big(f_{257}\Big)' e^{257}-\Big(f_{236}\Big)' e^{236}+\Big(f_{46}h_1-f_{16}h_3\Big)' e^{146}Ê\\
&-\Big(f_{45}h_2+f_{25}h_3\Big)' e^{245}+\Big(-f_{27}h_1+f_{17}h_2\Big)' e^{127}.
\end{aligned}
\end{equation}
On the other hand, 
\begin{equation*}
\begin{aligned} \label{eq-ast-N4}
*_{t}\varphi_{4}(t)=x^{3567}+x^{1237}+x^{1256}-x^{2467}+x^{2345}+x^{1457}+x^{1346}.
\end{aligned}
\end{equation*}
 So, $x^{3567}$ and $x^{2467}$ are the nonclosed summands in $*_{t}\varphi_{4}(t)$. 

Then, for $\Delta_{t}\varphi_{4}(t)=-d *_{t}d *_{t}\varphi_{4}(t)$ we obtain
\begin{equation}\label{eq-Lap-flow:N4-7}
\begin{array}{lcl}
\Delta_{t}\varphi_{4}(t)&=&-\Big(f_{124}(\frac{f_{3}^2}{f_{1}^2 f_{2}^2}+\frac{f_{6}^2}{f_{2}^2f_{4}^2})
-\frac{f_{37}h_{3}}{f_{15}}-\frac{f_{6}^2h_{1}}{f_{13}}\Big)e^{124}\\
&&+f_{135}\Big(\frac{f_{6}^2}{f_{1}^2f_{3}^2}+\frac{f_{7}^2}{f_{1}^2f_{5}^2}\Big)e^{135} 
+\frac{f_{5}f_{6}^2}{f_{13}}e^{245} + \frac{f_{3}f_{7}^2}{f_{1}f_{5}}e^{127}.
\end{array}
\end{equation}
Comparing \eqref{eq-Lap-flow:N4-5} and \eqref{eq-Lap-flow:N4-7} we see that the functions
$f_i$, $h_1$ and $h_3$ satisfy
\begin{equation*}
f_{167}(t)=f_{236}(t)=f_{257}(t)=f_{347}(t)=f_{456}(t)=1,  \qquad  f_{46}(t)h_{1}(t)-f_{16}(t)h_{3}(t)=0,
\end{equation*}
for any $t\in I$. But these equations are satisfied if
\begin{equation}\label{first-sol-N4}
f_1=f_{23}^2, \qquad f_4=f_2, \qquad f_5=f_3, \qquad f_6=f_7=\frac{1}{f_{23}},  \qquad 
h_1=f_2f_{3}^2h_3.
\end{equation}
Using \eqref{first-sol-N4}, we write \eqref{eq-Lap-flow:N4-5} and \eqref{eq-Lap-flow:N4-7}
in terms of $f_i$, $h_1$ and $h_3$. Then, we see that $\frac{d}{dt}\varphi_{4}(t)=\Delta_{t}\varphi_{4}(t)$
if and only if  
\begin{equation}\label{second-sol-N4}
\begin{array}{l} 
f_1=u\cdot v,  \quad f_2=f_4=v^{1/2},  \quad f_3=f_5=u^{1/2},  \quad 
f_6=f_7=(uv)^{-1/2}, \\
h_1=\frac{1}{2}u^{5/2}v-\frac12u^{1/2}, \quad h_2=0,  \quad 
h_3=\frac12 u^{3/2} v^{1/2}-\frac12(uv)^{-1/2},
\end{array}
\end{equation}
where $u=u(t)$ and $v=v(t)$ are differentiable real functions satisfying the 
 system of ordinary differential equations \eqref{eqn:1-N4} with initial conditions \eqref{eqn:2-N4}.
By Corollary \ref{interval:N4}, we know that the system \eqref{eqn:1-N4}-\eqref{eqn:2-N4} has a 
solution $u=u(t), v=v(t)$
defined in  $I=(t_{min},+\infty)$.
Then, taking into account
 \eqref{sol-N4-1} and \eqref{second-sol-N4}, the family of closed $G_2$ 
 forms $\varphi_{4}(t)$ solving \eqref{eq-Lap-flow:2} for $\varphi_4$ is given by
\begin{align*}
\varphi_{4}(t)=&\frac{1}{4} e^{124} \left(-u^4v^2+2 u^2 v-4 u v^2-1\right)+\frac{1}{2} e^{127} \left(u^2 v-1\right)+
u^2 v e^{135}\\
&+e^{167}-e^{236}+\frac{1}{2} e^{245} \left(u^2v-1\right)+e^{257}+e^{347}-e^{456},
   \end{align*}
for $t \in (t_{min}, +\infty)$.
The underlying metric $g(t)$ of  this  solution converges 
to  a  flat metric.
To check that the corresponding manifold in the limit is flat, we note that  all non-vanishing coefficients 
of the Riemannian curvature  $R(t)$ of $g(t)$ are proportional to the function $2u(t)-u^4(t)$. According with 
Corollary \ref{interval:N4}),
we have that the function $u(t)$ satisfies
$$\lim_{t\rightarrow + \infty}u(t)=2^{1/3},$$
and so
$$\lim_{t\rightarrow + \infty} R(t)=0.$$ \end{proof}

Concerning the Laplacian flow \eqref{eq-Lap-flow:2}
of the closed $G_2$ form $\varphi_6$ on $N_6$ we have the 
following.
\begin{theorem}\label{Lap-flow: N6}
The Laplacian flow of $\varphi_6$ has a
solution $\varphi_{6}(t)$ on $N_6$ defined in the interval
$I=(t_{min},+\infty)$,
where $t_{min}$ is the negative real number given by \eqref{exprtmin}.
Moreover, the underlying metrics  $g(t)$ of this solution converge smoothly, up to pull-back by time-dependent diffeomorphisms,  to a flat metric, uniformly on compact sets in $N_6$, as $t$ goes to infinity.
\end{theorem}
\begin{proof}
We take differentiable real functions
$f_i=f_{i}(t)$ $(i=1,\dots,7)$ and $h_{j}=h_{j}(t)$ $(j=1,2)$ depending on a parameter 
$t\in I\subset{\mathbb{R}}$ such that $f_{i}(0)=1, h_{j}(0)=0$ and $f_{i}(t)\not=0$, for any $t\in I$
and for any $i$ and $j$.
Now, for each $t\in I$, we consider the basis $\{x^1,\dots,x^7\}$ of left invariant $1$-forms
on $N_6$ defined by
\begin{equation*}
\begin{aligned}
x^i&=x^i(t)=f_{i}(t)e^i,\quad 1\leq i\leq 5, \\
x^6&=x^{6}(t)=f_6(t)e^6+h_{1}(t)e^2,  \\
x^7&=x^{7}(t)=f_7(t)e^7+h_{2}(t)e^3. 
\end{aligned}
\end{equation*}
For any $t\in I$, let $\varphi_{6}(t)$ the $G_2$ form on $N_6$ defined by
\begin{equation}
\label{sol-N6-1}
\varphi_{6}(t)=x^{123}+x^{145}+x^{167} + x^{257} - x^{246}+x^{347}+x^{356}.
\end{equation}

In order to study the flow \eqref{eq-Lap-flow:2} of $\varphi_6$, we proceed
as in the proof of Theorem \ref{Lap-flow: N4}. 
We see that the forms $\varphi_{6}(t)$ defined by \eqref{sol-N6-1} are
a solution of \eqref{eq-Lap-flow:2} if and only if the functions
$f_i$, $h_1$ and $h_2$ satisfy
\begin{equation*}\label{second-sol-N6}
\begin{aligned}
f_1=u\cdot v,  \qquad f_2=f_3=v^{1/2},  \qquad f_4=f_5=u^{1/2},  \\
f_6=f_7=(uv)^{-1/2}, \qquad h_1=h_2=-\frac{1}{2}(uv)^{-1/2}+\frac12u^{3/2}v^{1/2},
\end{aligned}
\end{equation*}
where $u=u(t)$ and $v=v(t)$ are differentiable real functions satisfying the 
system of ordinary differential equations
\begin{equation}\label{eqn:1-N6}
\begin{cases}
u'=\dfrac 23\,\dfrac{2-u^3}{u^3v^3},\\ v'=-\dfrac 23\,\dfrac{1-2u^3}{u^4v^2},\end{cases}
\end{equation} 
with initial conditions
\begin{equation}\label{eqn:2-N6}
u(0)=v(0)=1.
\end{equation}

Clearly, the systems \eqref{eqn:1-N6}-\eqref{eqn:2-N6} and 
\eqref{eqn:1-N4}-\eqref{eqn:2-N4} are the same.  
Thus, the maximal solution of \eqref{eqn:1-N6}-\eqref{eqn:2-N6}   
satisfies the properties expressed in Corollary \ref{interval:N4}
for the maximal solution of \eqref{eqn:1-N4}-\eqref{eqn:2-N4}.

To finish the proof we see that, for $t \in (t_{min}, +\infty)$,
the expression of $\varphi_{6}(t)$ is given by
\begin{align*}
\varphi_{6}(t)=&\frac14(1+4uv^2-2u^2v+u^4v^2)e^{123}+e^{347}+e^{356}+e^{167}-e^{246}+e^{257} \\
&+u^2ve^{145}+\frac12(1-u^2v)(e^{136}-e^{127}).
\end{align*}
The underlying metric $g(t)$ of this  solution converges 
to  a  flat metric.
To check that the limit metric is flat, we note that  all non-vanishing coefficients 
of the Riemannian curvature  $R(t)$ of $g(t)$ are proportional to the function $$u^p(t)(2-u^3(t))^q,$$ where $p$ and $q$ are real numbers satisfying that $q >0$. According with 
Corollary \ref{interval:N4}),
we have that the function $u(t)$ satisfies
$$\lim_{t\rightarrow + \infty}u(t)=2^{1/3},$$
and so
$$\lim_{t\rightarrow + \infty} R(t)=0.$$
\end{proof}
\begin{remark} Note that  surprising in the $N_4$ and  $N_6$ cases we get the same system of equations.  \end{remark}

 Finally, for the Laplacian flow of the closed $G_2$ form $\varphi_{12}$ on $N_{12}$
we have the following.
\begin{theorem}\label{Lap-flow: N12} The family of closed $G_2$ forms $\varphi_{12}(t)$ on $N_{12}$ given by
\begin{equation}\label{varphi12t}
\varphi_{12}(t)=-e^{124}+e^{167}+f(t)^6 e^{135}-f(t)^6 e^{236}+e^{257}+e^{347}-e^{456}, \, \,  t\in ( -3, +\infty)
\end{equation}
is the solution of the Laplacian flow of $\varphi_{12}$, where $f=f(t)$ is the function 
\begin{equation*}
f(t)=\Big(\frac13t+1\Big)^{1/8}.
\end{equation*}
Moreover,  the underlying metrics  $g(t)$ of this solution converge smoothly, up to pull-back by time-dependent diffeomorphisms,  to a flat metric, uniformly on compact sets in $N_{12}$, as $t$ goes to infinity. 
\begin{proof} Let $f_i=f_i(t)$ $(i=1,\dots, 7)$ be some differentiable real functions depending on a parameter $t\in I\subset \mathbb{R}$ such that $f_i(0)=1$ and $f_i(t)\neq 0$, for any $t \in I$, where $I$ is an open interval. For each $t\in I$, we consider the basis $\{x^1, \dots, x^7\}$ of left invariant 1-forms on $N_{12}$ defined by 
\begin{equation*}
x^i=x^i(t)=f_i(t)e^i, \hspace{0.2cm} 1\leq i \leq 7.
\end{equation*}
Then, from \eqref{n12} the structure equations of $N_{12}$ with respect to this basis are
\begin{equation}\label{equations_n12}
\begin{aligned}
&dx^i=0, \hspace{0.2cm} i=1,2,3, \hspace{0.2cm} \hspace{2.4cm} dx^4=\frac{\sqrt{3}}{6}\frac{f_4}{f_{12}}x^{12},\\
&dx^5=-\frac{1}{4}\frac{f_5}{f_{23}}x^{23}+\frac{\sqrt{3}}{12}\frac{f_5}{f_{13}}x^{13}, \hspace{0.8cm} dx^6=-\frac{\sqrt{3}}{12}\frac{f_6}{f_{23}}x^{23}-\frac{1}{4}\frac{f_6}{f_{13}}x^{13},\\
&dx^7=-\frac{\sqrt{3}}{6}\frac{f_7}{f_{34}}x^{34}+\frac{\sqrt{3}}{12}\frac{f_7}{f_{25}}x^{25}+\frac{1}{4}\frac{f_7}{f_{26}}x^{26}+\frac{\sqrt{3}}{12}\frac{f_7}{f_{16}}x^{16}-\frac{1}{4}\frac{f_7}{f_{15}}x^{15}.
\end{aligned}
\end{equation}
Now, for any $t\in I$, we consider the $G_2$ form $\varphi_{12}(t)$ on $N_{12}$ given by 
\begin{equation}\label{phi12}
\begin{aligned}
\varphi_{12}(t)&=-x^{124}+x^{167}+x^{135}-x^{236}+x^{257}+x^{347}-x^{456}=\\
&=-f_{124}e^{124}+f_{167}e^{167}+f_{135}e^{135}-f_{236}e^{236}+f_{257}e^{257}+f_{347}e^{347}-f_{456}e^{456}.
\end{aligned}
\end{equation}
Note that $\varphi_{12}(0)=\varphi_{12}$ and, for any $t$, the 3-form $\varphi_{12}(t)$ on $N_{12}$ determines the metric $g_t$ such that the basis $\{x_i=\frac{1}{f_1}e_i; i=1,\dots, 7\}$ of $\mathfrak{n}_{12}$ is orthonormal. So, $g_t(e_i, e_i)=f_i^2$. \\
We need to determine the functions $f_i$ and the interval $I$ so that $\frac{d}{dt}\varphi_{12}(t)=\Delta_t \varphi_{12}(t)$, for $t\in I$. Using \eqref{phi12} we have 
\begin{equation}\label{dtvarphi12}
\begin{aligned}
\frac{d}{dt}\varphi_{12}(t)=&-(f_{124})'e^{124}+(f_{167})'e^{167}+(f_{135})'e^{135}-(f_{236})'e^{236}+\\
&+(f_{257})'e^{257}+(f_{347})'e^{347}-(f_{456})'e^{456}.
\end{aligned}
\end{equation} 
Now, we calculate $\Delta_t\varphi_{12}(t)=-d\ast_td\ast_t\varphi_{12}(t)$. On the one hand, we have
\begin{equation}\label{astvarphi12}
\ast_t\varphi_{12}(t)=x^{3567}-x^{2467}+x^{2345}+x^{1457}+x^{1346}+x^{1256}+x^{1237}.
\end{equation}
So, $x^{2467}$ and $x^{1457}$ are the unique non closed summands in $\ast_t\varphi_{12}(t)$. Then, taking into account the structure equations \eqref{equations_n12} and that $x^i(t)=f_i(t)e^i, 1\leq i \leq 7$ we obtain
\begin{equation}\label{Deltavarphi12}
\begin{aligned}
\Delta_t\varphi_{12}(t)=&-\frac{(f_{15}+f_{26})(f_5^2f_6^2+f_3^2f_7^2)}{16f_1f_2f_3f_5f_6}(e^{236}-e^{135})+\\
&+\frac{(f_{15}+f_{26})(f_5^2f_6^2-f_3^2f_7^2)}{16\sqrt{3}f_1f_2f_3f_5f_6}(e^{136}+e^{235}).
\end{aligned}
\end{equation} 
Comparing \eqref{dtvarphi12} and \eqref{Deltavarphi12}, in particular, we have that 
\begin{equation*}
(f_{124})'=(f_{167})'=(f_{257})'=(f_{347})'=(f_{456})'=0,
\end{equation*}
and since $\varphi_{12}(0)=\varphi_{12}$ this imply that 
\begin{equation}\label{identidades}
f_{124}(t)=f_{167}(t)=f_{257}(t)=f_{347}(t)=f_{456}(t)=1,
\end{equation}
for any $t\in I$.
From the equation \eqref{identidades} we obtain that 
\begin{equation*}
f_1=f_1; \quad f_2=f_2;  \quad f_3=(f_1f_2)^2; \quad f_4=\frac{1}{f_1f_2}; \quad f_5=f_1; \quad f_6=f_2; \quad f_7=\frac{1}{f_1f_2}.
\end{equation*}
Let us consider $f=f_1=f_2$. With these concrete values \eqref{dtvarphi12} and \eqref{Deltavarphi12} become
\begin{equation}\label{dtvarphi12new}
\frac{d}{dt}\varphi_{12}(t)=(f^6(t))'(e^{135}-e^{236}),
\end{equation}
and
\begin{equation}\label{Deltavarphi12new}
\Delta_t\varphi_{12}(t)=\frac{f(t)^{-2}}{4}(e^{135}-e^{236}), 
\end{equation}
respectively.
From \eqref{dtvarphi12new} and \eqref{Deltavarphi12new} finding a solution of the Laplacian flow is equivalent to solve $f^7f'=\frac{1}{24}$. Integrating this equation, we obtain 
\begin{equation*}
f^8=\frac13t+B, \quad B=constant.
\end{equation*}
But $\varphi(0)=\varphi_{12}$ implies that $f(0)=1$, that is, $B=1$. Hence 
\begin{equation*}
f(t)=\Big(\frac13t+1\Big)^{1/8},
\end{equation*}
and so the one parameter family of 3-forms $\{\varphi_{12}(t)\}$ given by \eqref{varphi12t} is the solution of the Laplacian flow of $\varphi_{12}$ on $N_{12}$, and it is defined for every $t \in (-3, +\infty)$.\\
Finally, we study the behavior of  the underlying metric $g(t)$ of such a solution in the limit.
If we think of  the Laplacian flow  as a one parameter  family of $G_2$ manifolds with a closed $G_2$-structure, it can also be checked that, in the limit, the resulting manifold has vanishing curvature.   Denote by $g(t)$, $t\in \left (-3,+ \infty \right)$, the metric on $N_{12}$ induced by the $G_2$ form $\varphi_{12}(t)$ defined by \eqref{varphi12t}. Then, $g(t)$ has the following expression
\begin{equation*}
\begin{aligned}
g(t)&=\Big(\frac{1}{3}t+1\Big)^{1/4} (e^1)^2 +\Big(\frac{1}{3}t+1\Big)^{1/4} (e^2)^2 +\Big(\frac{1}{3}t+1\Big)^{-1} (e^3)^2
\\&+\Big(\frac{1}{3}t+1\Big)^{-1/2} (e^4)^2 +\Big(\frac{1}{3}t+1\Big)^{1/4} (e^5)^2 +\Big(\frac{1}{3}t+1\Big)^{1/4} (e^6)^2 
\\&+\Big(\frac{1}{3}t+1\Big)^{-1/2} (e^7)^2.
\end{aligned}
\end{equation*}
  Concretely, every non vanishing coefficient appearing in the expression of the Riemannian curvature $R(t)$  of $g(t)$ is proportional to $(t+3)^{-1}$. Therefore, $\lim_{t \rightarrow +\infty} R(t)=0$.
\end{proof}
\end{theorem}

\begin{remark}
Note that, for every $t \in \left (-3,+ \infty \right)$, the metric $g(t)$ is a nilsoliton on the Lie algebra $\frak n_{12}$ of $N_{12}$.
In fact,  with respect to the orthonormal basis $(x_1(t), \ldots, x_7 (t))$, we have \begin{equation*} Ric(g(t))  = - \frac{3}{4(3+t)} Id +  \frac{3}{8(3+t)} diag(1,1,1,2,2,2,3) =
\frac{3}{(3+t)} Ric(g(0))\end{equation*} with  $ \frac{3}{8(3+t)} diag(1,1,1,2,2,2,3)$ a derivation of $\frak n_{12}$ for every $t$.
\end{remark}

\end{section}

\bigskip

\noindent {\bf Acknowledgments.} We would like to thank Vivina Barutello, Ernesto Buzano, Diego Conti, Edison Fern\'andez, Jorge Lauret  and Mario Valenzano  for useful
conversations. Moreover, we are grateful to the anonymous referees for useful comments and improvements.  This work has been partially supported
by (Spanish) MINECO Project MTM2011-28326-C02-02,
Project UPV/EHU ref.\ UFI11/52 and by (Italian) GNSAGA of INdAM.



\bigskip

\small\noindent Universidad del Pa\'{\i}s Vasco, Facultad de Ciencia y Tecnolog\'{\i}a, Departamento de Matem\'aticas,
Apartado 644, 48080 Bilbao, Spain. \\
\texttt{marisa.fernandez@ehu.es}\\
\texttt{victormanuel.manero@ehu.es}

\bigskip

\small\noindent Dipartimento di Matematica, Universit\`a di
Torino, Via Carlo Alberto 10, Torino, Italy.\\
\texttt{annamaria.fino@unito.it}

\end{document}